\documentclass[a4paper,11pt]{amsart}

\usepackage{amscd}
 \usepackage{xypic}  %commu
\usepackage{amssymb}
\usepackage{amsthm}
\usepackage{epsfig}
\usepackage[T1]{fontenc}
\usepackage{color}
\usepackage{verbatim}
\usepackage{appendix}

\newtheorem{thm}{Theorem}[section]
\newtheorem{cor}[thm]{Corollary}
\newtheorem{lemma}[thm]{Lemma}
\newtheorem{prop}[thm]{Proposition}

\newtheorem{claim}[thm]{Claim}
\renewcommand{\proofname}{Proof}
\newtheorem{proposition}[thm]{Proposition}

\newtheorem{question}[thm]{Question}

\theoremstyle{definition}
\newtheorem{remark}[thm]{Remark}

\newtheorem{definition}[thm]{Definition}

 \newtheorem{example}[thm]{Example}

\def\Exc{\operatorname{Exc}}

\def\coker{\operatorname{coker}}
\def\codim{\operatorname{codim}}
\def\min{\operatorname{min}}
\def\im{\operatorname{im}}

\def\max{\operatorname{max}}

\def\c1{\operatorname{c_1}}
\def\c2{\operatorname{c_2}}
\def\Cliff{\operatorname{Cliff}}
\def\gon{\operatorname{gon}}

\def\Hilb{\operatorname{Hilb}}

\def\Sym{\operatorname{Sym}}

\def\rk{\operatorname{rk}}

\def\expdim{\operatorname{expdim}}
\def\NS{\operatorname{NS}}
\def\gg{\mathfrak{g}}
\def\CC{{\mathbb C}}
\def\ZZ{{\mathbb Z}}

\def\DD{{\mathbb D}}

\def\PP{{\mathbb P}}
\def\A{{\mathcal A}}

\def\SS{{\mathcal S}}
\def\G{{\mathcal G}}

\def\L{{\mathcal L}}
\def\M{{\mathcal M}}
\def\N{{\mathcal N}}
\def\O{{\mathcal O}}

\def\E{{\mathcal E}}

\def\H{{\mathcal H}}
\def\F{{\mathcal F}}
\def\K{{\mathcal K}}

\def\V{{\mathcal V}}

\def\C{{\mathcal C}} 
\def\K{{\mathcal K}}

\def\Q{{\mathcal Q}}

\def\x{\times}                   % product (fiber)
\def\cong{\simeq}

\def\sub{\subseteq}

\def\+{\oplus}                   % direct sum
\def\*{\otimes}                  % tensor product
\def\Pic{\operatorname{Pic}}
\def\Supp{\operatorname{Supp}}

\def\Supp{\operatorname{Supp}}

\def\Sing{\operatorname{Sing}}

\begin{document}

\title{Severi Varieties and Brill-Noether theory of curves on abelian surfaces}

\author[A.~L.~Knutsen]{Andreas Leopold Knutsen}
\address{Andreas Leopold Knutsen, Department of Mathematics, University of Bergen,
Postboks 7800,
5020 Bergen, Norway}
\email{andreas.knutsen@math.uib.no}

\author[M.~Lelli-Chiesa]{Margherita Lelli-Chiesa}
\address{Margherita Lelli-Chiesa, Centro di Ricerca Matematica ``Ennio De Giorgi'', Scuola Normale Superiore, Piazza dei Cavalieri 3, 56100 Pisa, Italy}
\email{margherita.lellichiesa@sns.it}

\author[G.~Mongardi]{Giovanni Mongardi}
\address{Giovanni Mongardi, Department of Mathematics, University of Milan, via Cesare Saldini 50, 20133 Milan, Italy}
\email{giovanni.mongardi@unimi.it}

\begin{abstract} 
Severi varieties and Brill-Noether theory of curves on K3 surfaces are well understood. Yet, quite little is known for curves %lying 
on abelian surfaces. Given a general abelian surface $S$ with polarization $L$ of type $(1,n)$, we prove nonemptiness and regularity of the Severi variety parametrizing $\delta$-nodal curves  in the linear system $|L|$ for $0\leq \delta\leq n-1=p-2$ (here $p$ is the arithmetic genus of any curve in $|L|$). We also show that a general genus $g$ curve having as nodal model a hyperplane section of some $(1,n)$-polarized abelian surface admits only finitely many such models up to translation; moreover, any such model lies on finitely many $(1,n)$-polarized abelian surfaces. 
Under certain assumptions, a conjecture of Dedieu and Sernesi is proved concerning the possibility of deforming a genus $g$ curve in $S$ equigenerically to a nodal curve.  The rest of the paper deals with the Brill-Noether theory of curves in $|L|$. It turns out that a general curve in $|L|$ is Brill-Noether general. However, as soon as the Brill-Noether number is negative and some other inequalities are satisfied, the locus $|L|^r_d$ of smooth curves in $|L|$ possessing a $g^r_d$ is nonempty and has a component of the expected dimension. As an application, we obtain the existence of a component of the Brill-Noether locus $\M^r_{p,d}$ having the expected codimension in the moduli space of curves $\M_p$. For $r=1$, the results are generalized to nodal curves.
\end{abstract}

\maketitle
\section{Introduction}
Sections of $K3$ surfaces have been investigated at length; yet, very little is known about curves lying on an abelian surface. If $(S,L)$ is a general polarized $K3$ surface, the proof of nonemptiness and regularity of the Severi variety parametrizing $\delta$-nodal curves in the linear system $|mL|$ with $0\leq \delta\leq\dim|mL|$ is due to Mumford for $m=1$ (cf. \cite[Appendix]{MM}), and Chen \cite{Ch} in the general case. The main result concerning the Brill-Noether theory of linear sections of $S$ is due Lazarsfeld \cite{La}, who proved that a general curve in the linear system $|L|$ is Brill-Noether general; furthermore, no curve in $|L|$ possesses any linear series with negative Brill-Noether number. Lazarsfeld's theorem provided an alternate proof of the Gieseker-Petri Theorem, thus highlighting the vast potential of specialization to $K3$ sections. Indeed, this technique proved useful in many contexts, such as Voisin's proof of Green's Conjecture for a general curve of any given genus \cite{V1,V2}, higher rank Brill-Noether theory \cite{FO,LC}, the proof of transversality of some Brill-Noether loci \cite{Fa,LC} and the study of rational curves on hyperk\"ahler manifolds of $K3^{[n]}$-type \cite{CK}. 

Our aim is to initiate the study of curves on abelian surfaces and provide a first application by exhibiting components of the Brill-Noether locus having the expected codimension in the moduli space of curves. Further applications to generalized Kummer manifolds will appear in \cite{KLM}. The main issue is that, unlike $K3$ surfaces, abelian surfaces are irregular; in particular, vector bundles techniques \`a la Lazarsfeld on abelian surfaces do not work as nicely as on $K3$ surfaces  (cf., e.g., Remark \ref{rem:optimusprime} and Example \ref{cex2}).

From now on, the pair $(S,L)$ will be a general $(1,n)$-polarized abelian surface; following \cite{LS}, we denote by $\{L\}$ the continuous system parametrizing curves in the linear system $|L|$ and in all of its translates by points of $S$, and by $p:=n+1$ the arithmetic genus of any curve in $\{L\}$. 

For fixed $\delta\geq 0$, we consider the Severi varieties $V_{|L|,\delta}(S)$ and $V_{\{L\},\delta}(S)$ parametrizing integral $\delta$-nodal curves in $|L|$ and $\{L\}$, respectively; they are locally closed in $|L|$ (resp. $\{L\}$) and have expected codimension  $\delta$.  
Recently, Dedieu and Sernesi \cite{DS} proved that any integral curve $C\in\{L\}$ deforms to a nodal curve of the same geometric genus, unless the normalization of $C$ is trigonal. However, the nonemptiness problem for $V_{\{L\},\delta}(S)$ is 
still open and we solve it by proving the following:

\begin{thm} \label{thm:main1}
  Let $(S,L)$ be a general polarized abelian surface with $L$ of type $(1,n)$. Then, for any integer $\delta$ such that
$0 \leq \delta \leq p-2=n-1$, the Severi variety 
$V_{\{L\},\delta}(S)$ (respectively, $V_{|L|,\delta}(S)$) is nonempty and smooth of dimension $p-\delta$ (resp., $p-\delta-2$). 
\end{thm}

Note that the bound on $\delta$ is also necessary since $\dim |L|=p-2$. The condition of being smooth of the expected dimension is often referred to as being {\it regular}.

Given Theorem \ref{thm:main1}, it is                                        
natural to investigate the variation in moduli of nodal curves lying on abelian surfaces.
In other words, one is interested in the dimension of the locus $\mathfrak{A}_{g,n}$ parametrizing curves in $\M_g$ that admit a nodal model (of arithmetic genus $p=n+1$) which is a hyperplane section of some $(1,n)$-polarized abelian surface. A simple count of parameters shows that this dimension is at most $g+1$ (cf. \S \ref{sec:modular}) and we prove that the bound is effective:

\begin{thm} \label{thm:moduli}
  For any $n \geq 1$ and $2 \leq g \leq p=n+1$, the  locus $\mathfrak{A}_{g,n}$ of curves in $\M_g$ admitting a $(p-g)$-nodal model as a hyperplane section of some $(1,n)$-polarized abelian surface has a component of dimension $g+1$. 

In particular, a general curve in such a component occurs as the normalization of a nodal hyperplane section of only finitely many $(1,n)$-polarized abelian surfaces. 
\end{thm}

Analogous results for smooth curves on $K3$ surfaces are due to Mori and Mukai
\cite{MM,mu0,Mu1} and there have been recent advances in the case of nodal curves \cite{Ke,CFGK}. On the contrary, nothing was known so far for (even smooth) curves lying on abelian surfaces, except in the principally polarized case. 

Another relevant question  is whether any genus $g$ curve in $\{L\}$ can be deformed equigenerically to a nodal curve  in $\{L\}$ (classically known for plane curves \cite{AC1,AC2,zar} and is currently being studied for curves on other surfaces \cite{DS}). An affirmative answer would imply that Severi varieties provide essential information about equigeneric families of curves in $\{L\}$. We show that this is indeed the case for $g\geq 5$, thus proving Conjecture C in \cite{DS} on a general $(1,n)$-polarized surface. 

\begin{thm} \label{thm:conjDS}
  Let $(S,L)$ be a general polarized abelian surface of type $(1,n)$. If $g \geq 5$, then the locus of curves in $\{L\}$ with geometric genus $g$ lies in the Zariski closure of the Severi variety $V_{\{L\},p-g}(S)$.
\end{thm}

The theorem also holds for $(S,L)$ a general primitively polarized $K3$ surface, cf. Remark \ref{rem:ancheK3}.

Henceforth, we focus on the Brill-Noether theory of curves in $|L|$. We denote by $|L|^r_d$ the Brill-Noether loci parametrizing smooth curves $C\in |L|$ carrying a linear series of type $g^r_d$. 

\begin{thm} \label{thm:main4}
Let $(S,L)$ be a general polarized abelian surface with $L$ of type $(1,n)$, and fix integers $r\geq 1$ and $d\geq 2$. Then the following hold: 
  \begin{itemize}
  \item[(i)] for a general $C\in |L|$, the Brill-Noether variety $G^r_d(C)$ is equidimensional of dimension $\rho(p,r,d)$ if $\rho(p,r,d)\geq 0$, and empty otherwise;
  \item[(ii)] if $d\geq r(r+1)$ and $-r(r+2) \leq \rho(p,r,d) <0$, then the Brill-Noether locus $|L|^r_d$ has an irreducible component $Z$ of the expected dimension $p-2+\rho(p,r,d)$; furthermore, $Z$ can be chosen so that, if $C\in Z$ is general, then $G^r_d(C)$ has some $0$-dimensional components parametrizing linear series that define birational maps to $\mathbb{P}^r$ as soon as $r\geq 2$;
\item[(iii)] the locus $|L|^r_d$ is empty if $\rho(p,r,d)<-r(r+2)$.
\end{itemize}
\end{thm}

In particular, this proves that a general curve in the linear system is Brill-Noether general, in analogy with Lazarsfeld's famous result. However, contrary to the $K3$ case, there are 
subloci in $|L|$ which parametrize curves carrying linear series with negative Brill-Noether number. 
Paris \cite{Pa} had already  proved, by completely different methods (Fourier-Mukai transforms), parts (i), under the additional technical hypothesis $d \neq p-1$, and (iii)
(however, we obtain (iii) even for torsion free sheaves on singular curves, cf. Theorem \ref{thm:necex}, and a stronger bound in Theorem \ref{thm:necexforte}). We believe that the totally new result  (ii) concerning negative Brill-Noether numbers is the most striking. First of all, it prevents both the gonality and the Clifford index of smooth curves in $|L|$ from being constant (cf. Remark \ref{rem:nonconst}); this is a major difference between the abelian and the $K3$ world. Secondly and most importantly, it  has quite a strong implication for the geometry of the Brill-Noether loci 
$\M^r_{p,d}$ in the moduli space of smooth irreducible genus $p$ curves $\M_p$. In stating it, we let $\G^r_d$ be the scheme parametrizing pairs $([C], \gg)$ with $[C]\in \M_p$ and $\gg\in G^r_d(C)$, and denote by $\pi: \G^r_d\to \M_p$ the natural projection. 

\begin{thm} \label{thm:main5}
If $d\geq r(r+1)$ and $-r(r+2)\leq\rho(p,r,d)<0$, then the Brill-Noether locus $\M^r_{p,d}$ has an irreducible component $\M$ of the expected dimension $3p-3+\rho(p,r,d)$. Furthermore, $\M$ coincides with the image under $\pi$ of an irreducible component $\G$ of $\G^r_d$ of the same dimension as $\M$ such that, if $(C, \gg)\in \G$ is general and $r\geq 2$, then $\gg$ defines a birational map to $\mathbb{P}^r$.
\end{thm}

Steffen \cite{St} proved that, as soon as $\rho(p,r,d)<0$, the codimension in $\M_p$ of any component of $\M^r_{p,d}$ is bounded from above by $-\rho(p,r,d)$. However, the problem concerning nonemptiness of $\M^r_{p,d}$ and existence of components having the expected dimension is highly nontrivial and has a complete answer only for $r=1,2$. For arbitrary $r$ the picture is well-understood only if $\rho(p,r,d)=-1,-2$: in the former case $\M^r_{p,d}$ is an irreducible divisor \cite{EH4} and in the latter case every component of $\M^r_{p,d}$ has codimension $2$ \cite{Ed}. For very negative values of $\rho(p,r,d)$, there are plenty of examples of Brill-Nother loci exceeding the expected dimension, e.g., constructed by considering multiples of linear series or complete intersection curves. On the other hand, for slightly negative values of $\rho(p,r,d)$, Steffen's dimensional estimate is still expected to be effective \cite{EH4}. 
Sernesi \cite{Se}, followed by other authors, most recently Pflueger \cite{Pf}, proved the existence of components of $\M^r_{p,d}$ of the expected dimension under certain assumptions on $p,r,d$. However, (infinitely) many components detected by our Theorem \ref{thm:main5} were heretofore unknown (cf. Remark \ref{confronto}).

We spend a few more words on Theorem \ref{thm:main4}.
In the range $r < p-1-d$, the result is optimal, in the sense that (ii) and (iii) yield that $|L|^r_d \neq \emptyset$ if and only if $\rho(p,r,d) \geq -r(r+2)$. Indeed, the inequality $\rho(p,r,d)\geq -r(r+2)$ implies the condition 
$d \geq r(r+1)$ in (ii) in precisely this range. 
On the other hand, for $r \geq p-1-d$, the situation is more complicated and, as soon as $r\geq 2$, the condition $\rho(p,r,d) \geq -r(r+2)$ is no longer sufficient for nonemptiness of $|L|^r_d$ (cf. Examples \ref{cex1}--\ref{cex2}). In fact, we prove a stronger necessary condition in the Appendix (cf. Theorem \ref{thm:necexforte}). 

For $r=1$, Theorem \ref{thm:main4} is clearly optimal and can be generalized to {\it nodal} curves. Let $|L|^1_{\delta,k}$ be the Brill-Noether locus parametrizing nodal curves $C\in V_{|L|,\delta}(S)$ such that the normalization of $C$ carries a $g^1_k$. We prove the following:

\begin{thm} \label{thm:main2}
 Let $(S,L)$ be a general  polarized abelian surface with $L$ of type $(1,n)$. Let $\delta$ and $k$ be integers satisfying $0 \leq \delta \leq p-2=n-1$ and $k \geq 2$, and set $g:=p-\delta$.  Then the following hold:
  \begin{itemize}
\item [(i)] $|L|^1_{\delta,k}(S) \neq \emptyset$ if and only if 
\begin{equation} \label{eq:boundintro}
\delta \geq \alpha\Big(p-\delta-1-k(\alpha+1)\Big), \; \; \mbox{with} \; \; \alpha= \Big\lfloor \frac{g-1}{2k}\Big\rfloor; 
\end{equation}
\item [(ii)]  whenever nonempty, $|L|^1_{\delta,k}$ is equidimensional of dimension  $\min\{g-2,2(k-2)\}$ and a  general element in each component is an irreducible curve $C$ with normalization $\widetilde{C}$ of genus $g$ such that $\dim G^1_k(\widetilde{C})=\max\{0,\rho(g,1,k)=2(k-1)-g\}$;
\item [(iii)]  there is at least one component of $|L|^1_{\delta,k}$ where, for $C$ and $\widetilde{C}$ as in (ii), the Brill-Noether variety $G^1_k(\widetilde{C})$ is reduced; furthermore, when $g \geq 2(k-1)$ (respectively $g < 2(k-1)$), any (resp. a general) $g^1_k$ on $\widetilde{C}$ has simple ramification and all nodes of $C$ are non-neutral with respect to it.
\end{itemize}
\end{thm}

This means that, for fixed $\delta \geq 0$, the gonality of the normalization of a general curve in $V_{|L|,\delta}(S)$ is that of a general genus $g$ curve, i.e.,  $\lfloor(g+3)/2\rfloor$. However, for all $k$ satisfying \eqref{eq:boundintro},  there are $2(k-1)$-dimensional subloci in $V_{|L|,\delta}(S)$ of curves whose normalization has lower gonality $k$. Parts (ii) and (iii) of  Theorem \ref {thm:main2} 
imply that $\widetilde{C}$ enjoys properties of a general curve of gonality $\min\{k,\lfloor(g+3)/2\rfloor\}$
with respect to pencils. 

Following ideas of \cite{CK}, Theorem \ref{thm:main2} will be used in \cite{KLM} in order to construct rational curves in the {\it generalized Kummer variety} $K^{[k-1]}(S)$, and item (iii) will be relevant in this setting, e.g., in the computation of the class of the rational curves. By (ii), these curves will move in a family of dimension precisely $2k-4$. This is the expected dimension of any family of rational curves on a $(2k-2)$-dimensional hyperk\"ahler manifold  \cite{ran},  whence (cf. %\cite[Prop. 3.1]{cp}, 
\cite[Pf. of Cor. 4.8]{AV}) the constructed families of rational curves deform to a general projective $(2k-2)$-dimensional hyperk\"ahler manifold deformation equivalent to a generalized Kummer variety and are therefore of particular interest. 

Finally, we remark that the dimensional statement in (ii) extends to curves of geometric genus $g$ in $|L|$ with arbitrary singularities (cf. Theorem \ref{thm:boundingdim}), as well as to $K3$ surfaces (cf. Remark \ref{rem:ancheK3}).

\subsection{Methods of proof}

Most results are proved by degeneration 
to a $(1,n)$-polarized semiabelian surface $(S_0,L_0)$, which is constructed starting with a ruled surface $R$ over an elliptic curve $E$ by identifying two sections $\sigma_\infty$ and $\sigma_0$ with a translation by a fixed nontorsion point $e\in E$, as in \cite{HW}.

The proof of  Theorem \ref{thm:main1} relies on the construction of curves in $S_0$ that are limits of $\delta$-nodal curves on smooth abelian surfaces, which becomes a simple combinatorial problem. 

Concerning Theorems \ref{thm:main4} and \ref{thm:main2}, the statements yielding necessary conditions for nonemptiness of the Brill-Noether loci are based on variations of vector bundle techniques \`a la Lazarsfeld and follow from a more general result providing necessary conditions for the existence of torsion free sheaves on curves on abelian surfaces (cf. Theorem \ref{thm:necex} and its stronger version Theorem \ref{thm:necexforte} in the Appendix).

The remaining parts of Theorems \ref{thm:main4} and \ref{thm:main2} (except for the fact that item (ii) in the latter  holds on {\it every} component of $|L|^1_{\delta,k}$) are again proved by degeneration to $S_0$.  For $\delta=0$, the limit curves in $S_0$ are $n$-nodal curves $X\in |L_0|$ obtained from the elliptic curve $E$ along with $n=p-1$ points $P_1,\ldots, P_{n}\in E$, by identifying each $P_i$ with its $e$-translate. In order to prove Theorem \ref{thm:main2} for $\delta=0$ by degeneration, one has to determine whether such an $X$ lies in $\overline{\M}^1_{p,k}$; the theory of admissible covers  translates the problem into a question about the existence of a $g^1_k$ on $E$ identifying any pair of points corresponding to a node of $X$. This can be easily answered using intersection theory on the ruled surface $\Sym^2(E)$. The situation for $\delta >0$ is slightly more involved combinatorially, but the idea is the same. 

The proof of Theorem \ref{thm:main4} is more demanding. A first obstacle in understanding whether a limit curve $X\subset S_0$ as above defines a point of $\overline{\M}^r_{p,d}$ lies in the fact that linear series of type $g^r_d$ on smooth curves might tend to torsion free (not necessarily locally free) sheaves on $X$. Furthermore, even determining the existence of degree $d$ line bundles on $X$ with enough sections is hard. In order to cope with this problem, we further degenerate $X$ (forgetting the surface) to a curve $X_0$ with $p-1$ cusps, that is, we let $e$ approach the zero element of $E$. Line bundles on $X_0$ then correspond to linear series on $E$ having at least cuspidal ramification at $n$ points $P_1,\ldots, P_n$. The Brill-Noether Theorem with prescribed ramification \cite{EH1,EH2} at some general points then yields (i). In order to obtain (ii), we prove the existence of a family (having the expected dimension) of birational maps $\phi:E\to \mathbb{P}^r$ such that the images $\phi(E)$ are nondegenerate curves of degree $d$ with  $p-1$ cusps (corresponding to special points of $E$ as $\rho(p,r,d)<0$).  This is done by resorting to Piene's duality \cite{Pi} for nondegenerate curves in $\mathbb{P}^r$: with any nondegenerate curve $Y\subset \mathbb{P}^r$ of normalization $f:\widetilde{Y}\to Y$ one associates a dual curve $Y^\vee\subset (\mathbb{P}^r)^\vee$ defined as the image of the dual map $f^\vee:\widetilde{Y}\to (\mathbb{P}^r)^\vee$, which sends a point $P\in \widetilde{Y}$ to the osculating hyperplane of $Y$ at $f(P)$ (for $r>2$ the dual curve should not be confused with the dual hypersurface, which is the closure of the image of the Gauss map). Thanks to the duality theorem \cite[Thm. 5.1]{Pi}, one may try to construct the dual map $\phi^*$ and then recover $\phi$. This strategy proves successful since cusps of $Y$ transform into ordinary ramification points of $Y^\vee$, hence the existence of $\phi^*$ can be more easily achieved than that of $\phi$.

The proof of Theorem \ref{thm:moduli} does not use degeneration, but specialization to the smallest Severi varieties (namely, those parametrizing curves of geometric genus $2$), where the result is almost trivial. Since these are contained in the closure of the Severi varieties parametrizing curves with fewer nodes, it is not hard to deduce the result in general. The  universal Severi variety over a suitable cover of the moduli space of $(1,n)$-polarized  abelian surfaces plays an important role.

Finally, Theorem \ref{thm:conjDS}, as well as Theorem \ref{thm:main2}(ii) for {\it all} components of the Brill-Noether loci, follow from a bound on the dimension of any family of curves with arbitrary singularities whose normalizations possess linear series of type $g^1_k$ (cf. Theorem \ref{thm:boundingdim}). This is obtained by bounding the corresponding family of rational curves in the generalized Kummer variety $K^{[k-1]}(S)$ through Mori's bend-and-break and recent results by Amerik and Verbitsky \cite{AV}. This is a nice sample of the rich interplay between the theory of abelian (and $K3$) surfaces and hyperk\"ahler manifolds.

\subsection{Plan of the paper}

In \S \ref{degenerazioni} we introduce the degeneration used in the proof of the main results. In \S \ref{severi}, Severi varieties and their degenerations are investigated and Theorem \ref{thm:main1} is proved. Theorem \ref{thm:moduli} is obtained in \S \ref{sec:modular}. The rest of the paper is devoted to Brill-Noether theory. More precisely, in \S \ref{sec:lser} we compute the expected dimension of the Brill-Noether loci $|L|^r_{\delta,d}$ (cf. Proposition \ref{prop:expdimvk}) and provide a necessary condition for their nonemptiness (cf. Theorem \ref{thm:necex}). Furthermore, we bound the dimension of any family of curves in $|L|$ with arbitrary singularities such that their normalizations possess linear series of type $g^1_k$ (cf. Theorem \ref{thm:boundingdim}); this proves Theorem \ref{thm:conjDS}, by Dedieu and Sernesi's result mentioned above, and Theorem \ref{thm:main2}(ii). The rest of the proof of Theorem  \ref{thm:main2} unfolds \S \ref{sec:sevdeg}, while Theorems \ref{thm:main4} and \ref{thm:main5} are finally accomplished in \S \ref{cuspidi}. A stronger necessary condition for nonemptiness of $|L|^r_d$ in the range $r \geq p-1-d\geq 0$ is obtained in Theorem \ref{thm:necexforte} in
the Appendix. Both the statement and its proof are somewhat technical and are based on the fact that in this range the analogues of Lazarsfeld-Mukai bundles on $K3$ surfaces are forced to have nonvanishing $H^1$. This phenomenon does not occur on $K3$ surfaces and highlights the additional complexity of abelian surfaces.

\section{The degenerate abelian surface $(S_0,L_0)$}\label{degenerazioni}
In this section we introduce the degenerate abelian surface used in the proof of the main results. This degeneration is studied in \cite{HW} and \cite{HKW} when $n$ is either $1$ or an odd prime; we will extend the results necessary for our purposes to any integer $n\geq 1$.

\begin{definition} \label{def:dege}
  A proper flat family of surfaces $f : \SS \to \DD$ over the disc $\DD$ will be called a {\it degeneration of $(1,n)$-polarized abelian surfaces} if
\begin{itemize}
\item[(i)] $\SS$ is smooth;
\item[(ii)] the fiber $S_t$ over any $t \neq 0$ is a smooth abelian surface;
\item[(iii)] the fiber $S_0$ 
over $0$ is an irreducible surface with simple normal crossing singularities;
\item[(iv)] there is a line bundle $\L$ on $\SS$ such that $\L_{|S_t}$ is a polarization of type $(1,n)$ for every $t \neq 0$.
\end{itemize}

The special fiber $S_0$ will be called a {\it degenerate $(1,n)$-polarized abelian surface} and the line bundle $L_0:=\L_{|S_0}$ a {\it limit of polarizations of type} $(1,n)$.
\end{definition}

We  fix notation. Given an elliptic curve $E$, we denote by $e_0$ the neutral element with respect to its group structure. To distinguish the sum in $E$ as a group from the sum of points in $E$ as divisors, we denote the former by $\+$ and the latter by $+$. The group law is defined in such a way that:
\begin{equation}\label{equation:law}
P+Q-e_0\sim P\+Q\,\,\,\textrm{ for all }P,Q\in E.
\end{equation}

We use the notation $e^{\oplus n}$, $n \in \ZZ$, $e \in E$, for the sum $e \+ \cdots \+e$ of $e$ repeated $n$ times, and $e^{\+ 0}=e_0$ by convention. We also write 
$e=e_1 \ominus e_2$ when $e_1=e \+ e_2$. Analogously, for any divisor $D=\sum n_iP_i$, $n_i \in \ZZ$, $P_i \in E$,  we define
$D \+ e:=\sum n_i(P_i \+ e)$.

\begin{proposition}\label{frittata}
Fix an integer $n\geq 1$ and let $E$ be the elliptic curve defined as
$E:=\CC/(\ZZ n+\ZZ \tau)$,
where $\tau$ lies in the upper complex half-plane $\mathbb{H}$. For a point $e\in E$, let $R$ denote the $\PP^1$-bundle $\PP(\O_E\oplus \N)$ over $E$, with $\N=\O_E(ne-ne_0)$. Then the surface $S_0$ obtained by gluing the section at infinity $\sigma_{\infty}$ and the zero-section $\sigma_0$ of $R$ with translation by $e\in E$ is a degenerate abelian surface. Furthermore, $S_0$ carries a line bundle $L_0$ which is the limit of polarizations of type $(1,n)$ and satisfies $\nu^*L_0\equiv \sigma_0+nF$, where $\nu:R\to S_0$ is the normalization map and $F$ denotes the numerical equivalence class of the fibers of $R$ over $E$.
\end{proposition}

\begin{proof}
Consider the codimension $1$ boundary component of the Igusa compactification $A^*(1,1)$ of the moduli space of principally polarized abelian surfaces $A(1,1)$ whose points correspond to period matrices of the form $\left(\begin{array}{llll}1&0&0&\tau_2\\0&1&\tau_2&\tau_3\end{array}\right)$, where $\tau_3\in\mathbb{H}$ and $\tau_2\in\CC$. By \cite[II,3; II,5B]{HKW}, there is a degenerate principally polarized abelian surface $A_0$  associated with such a matrix coinciding with the surface $S_0$ in our statement when $n=1$, $\tau:=\tau_3$ and $e:=[\tau_2] \in E$. We denote the corresponding family as in Definition \ref{def:dege} by $\varphi: \A \to \DD$ and the line bundle on $\A$ by $\L_\A$. One can choose $\DD$ such that $\A$ and $\A':=\varphi^{-1}(\DD \setminus \{0\})$ have the forms $\A'=G\backslash\widetilde{\A}'$ and $\A=G\backslash\widetilde{\A}$, where $G$ is a discrete group acting freely and properly discontinuously on some smooth, analytic spaces $\widetilde{\A}'$ and $\widetilde{\A}$ (cf. \cite[II, 3A]{HKW} and choose $\DD$ inside $U_{\sigma_{-1}}$).
The case $n=1$ now follows from  \cite[II, Prop. 5.18]{HKW}, where the class of $\nu^*L_0$ is computed. 

Let now $n>1$. Denote by $A_0$ the degenerate principally polarized abelian surface constructed starting from the elliptic curve $E':=\CC/(\ZZ+\ZZ \tau/n)$ and using as gluing parameter the point $e'$, which is the image of $e$ under the isomorphism $\alpha:E\to E'$ induced by $\alpha(z)=z/n$. We consider the degenerating family $\varphi:\mathcal{A}\to\mathbb{D}$ centered at $A_0$ constructed above. Up to restricting $\DD$ one finds a subgroup $G_n<G$ of finite index ($G_n\backslash G\simeq \ZZ_n$) such that $f':\mathcal{S}':=G_n\backslash\widetilde{\A}'\to \DD^*$ is a family of $(1,n)$-polarized abelian surfaces with a level structure. This reflects the fact that, given any abelian surface $S$ with polarization $L$ of type $(1,n)$ and having fixed a level structure on it, there exists a unique isogeny $\iota:S\to A$ with kernel isomorphic to $\ZZ_n$ such that $A$ is an abelian surface endowed with a principal polarization $L'$ and $\iota^*L'=L$. Since $G_n$ is a subgroup of $G$, it still acts freely and properly discontinuously on $\widetilde{\A}$ and the quotient space $\mathcal{S}=G_n\backslash\widetilde{\A}$ is smooth. We obtain in this way a degenerating family $f:\mathcal{S}\to\DD$ which extends $f'$ and, by construction, the fiber of $f$ over $0$ is the surface $S_0$ in the statement. If $\pi:\mathcal{S}\to\A$ is the map given by taking the quotient with respect to $G_n\backslash G$, the line bundle $L_0$ is obtained by restricting $\pi^*\L_\A$ to $S_0$.
\end{proof}

We make some remarks on $S_0$ and its normalization $R$. One has $K_R \sim -\sigma_0-\sigma_{\infty}$, whence $S_0$ has trivial dualizing bundle. The ruling of $R$ makes it possible to identify both $\sigma_{\infty}$ and $\sigma_0$ with $E$. On the singular surface $S_0$, a point $x\in \sigma_\infty\simeq E$ is identified with $x \+ e\in \sigma_0\simeq E$. It is easy to verify that $L_0^2=2n$.

Let $W:= \nu^*H^0(S_0,L_0) \subset H^0(R,\nu^* L_0)$ be the subspace of sections that are pullbacks of sections of $L_0$. Then  
$|W|:=\PP(W)=\nu^*|L_0| \subset |\nu^*L_0|$ is the linear subsystem of curves that are inverse images of curves in $|L_0|$ under the normalization $\nu$. These are curves $C \in |\nu^*L_0|$ satisfying
\begin{equation}
  \label{eq:cond}
 C \cap \sigma_{0} = (C \cap \sigma_{\infty}) \+e 
\end{equation}
as divisors on $E\cong \sigma_{\infty} \cong \sigma_0$. 

\begin{lemma} \label{lemma:rest}
The restriction map yields an isomorphism
  \begin{equation}
    \label{eq:rest}
\xymatrix{ 
r: H^0(\O_R(\nu^* L_0)) \ar[r]^{\hspace{-1.5cm}\cong} &  
H^0(\O_{\sigma_{\infty}}(\nu^* L_0)) 
\+ H^0(\O_{\sigma_0}(\nu^* L_0))
}   
\end{equation}
such that $s\in W$ if and only if $r(s)=(s_{\infty},s_0)$ with $s_0=T^*_e(s_{\infty})$, where $T_e$ denotes the translation by $e\in E$.
In particular, we obtain isomorphisms
 \begin{equation}
    \label{eq:rest2}
\xymatrix{
H^0(S_0, L_0)\cong   W  \ar[r]^{\stackrel{\Phi_e}{\cong}} &  
H^0\left(\O_{E}\left(\nu^* L_0\right)\right)}
\end{equation}
where $\Phi_e$ is the restriction map and its inverse is given by 
$s \mapsto r^{-1}\Big(s, T^*_e(s)\Big)$
\end{lemma}

\begin{proof}
(See also the proofs of \cite[II, Prop.~5.45]{HKW} and \cite[Prop.~4.2.5]{HW}.)
Since $\nu^* L_0-\sigma_{\infty}-\sigma_0 \equiv -\sigma_{\infty}+nF \equiv K_R + \sigma_0+nF$, we have $h^i(\nu^* L_0-\sigma_{\infty}-\sigma_0)=0$ for $i=0,1$, which yields \eqref{eq:rest}. The claim about the sections in $W$ is just a reformulation of condition \eqref{eq:cond}. The isomorphisms in (\ref{eq:rest2}) are then immediate.
\end{proof}

The lemma yields in particular that $\dim |W|= \deg \O_{E}(\nu^* L_0)-1=n-1=\dim |L_t|$. 

\begin{remark}
As in \cite[Thm. 4.2.6]{HW}, one can show that, as soon as $n \geq 5$ and $e\in E$ is not $2$-torsion,  the line bundle $L_0$ is very ample and defines an embedding of $S_0$ into the {\it translation scroll} 
$$S_0 := \cup_{x \in E} L(x,x \+e),$$ 
where $E\subset \PP^{n-1}$ is a rational normal curve. This surface has degree $2n$ and is smooth outside the double curve $E$. In this case, one has $L_0\simeq\O_{S_0}(1)$. 
\end{remark}

\section{Severi varieties on abelian surfaces}\label{severi}
This section is devoted to the study of limits of nodal curves on the degenerate abelian surface $(S_0,L_0)$. We will
prove Theorem \ref{thm:main1} and lay the ground for the proofs of the existence statements in Theorems \ref{thm:main2} and \ref{thm:main4}. First of all, we recall some background material on Severi varieties. Let $S$ be a projective irreducible surface with normal crossing singularities and let $|L|$ be a base point free,
complete linear system of Cartier divisors on $S$ whose general element is a connected curve $L$ of arithmetic genus $p=p_a(L)$ with at most
nodes as singularities, located at the singular locus  of $S$. Following \cite{LS}, we denote by $\{L\}$ the connected component of $\Hilb(S)$ containing $|L|$.

Let $\delta$ be a nonnegative integer. We denote by
$V_{|L|, \delta}(S)$ the locally closed 
subscheme of $|L|$ parametrizing curves $C\in \vert L\vert$ having only nodes as singularities, exactly $\delta$ of them (called the \emph{marked nodes}) off
the singular locus of $S$, and such that the partial normalization $\widetilde C$ at these 
$\delta$ nodes is connected (i.e., the marked nodes are \emph{non-disconnecting nodes}). We set $g:=p-\delta=p_a(\widetilde C)$. We likewise denote by 
$V_{\{L\}, \delta}(S)$ the analogous subscheme of
$\{L\}$.

As customary,  $V_{|L|, \delta}(S)$ and $V_{\{L\}, \delta}(S)$ will be called {\it regular} if they are smooth of codimension $\delta$ in $|L|$ and $\{L\}$, respectively.

If $S$ is smooth these varieties  are classically called {\em Severi
varieties} of $\delta$-nodal curves. We keep the same name in our more general setting. Many results proved for smooth $S$ go through. In particular, the proofs of Propositions 1.1 and 1.2 in \cite{LS} yield:

\begin{prop} \label{prop:rego}
If $\omega_S$ is trivial, then $V_{\{L\}, \delta}(S)$ is always regular when nonempty. 
\end{prop}

\begin{remark} \label{rem:sticazzi}
  If $S$ is a smooth abelian surface, then  $\{L\}$ is fibered over the dual surface $\hat{S}$ and each fiber is a translate of the central fiber $|L|$ (cf. \cite{LS}). This induces a corresponding fibration of 
$V_{\{L\}, \delta}(S)$ over  $\hat{S}$ having central fiber equal to $V_{|L|, \delta}(S)$; any other fiber is the Severi variety of $\delta$-nodal curves in a translate of $|L|$ and is naturally isomorphic to $V_{|L|, \delta}(S)$. In particular, the regularity of $V_{\{L\}, \delta}(S)$ (ensured by Proposition \ref{prop:rego}) implies that of $V_{|L|, \delta}(S)$ and all of its translates.
\end{remark}

 From now on, $(S,L)$ will be a general smooth $(1,n)$-polarized abelian surface. In order to study $V_{|L|, \delta}(S)$, we will investigate the Severi variety $V_{|L_0|, \delta}(S_0)$, where $(S_0,L_0)$ is the degenerate abelian surface from Proposition \ref{frittata}. Henceforth,  we assume that $e\in E$ is a nontorsion point. We set $$V_{|W|,\delta}(R):=\nu^*V_{|L_0|, \delta}(S_0)\subset |W|.$$ The elements of $V_{|W|,\delta}(R)$ parametrize curves in $|W|$ containing $\delta$ disjoint fibers algebraically equivalent to $F$, since $\nu^*L_0\cdot F=1$ with the same notation as in Proposition \ref{frittata}. Also note that if a curve in $|W|$ contains a fiber $F$ and $F \cap  \sigma_{\infty}=P$, then the gluing condition \eqref{eq:cond} forces the residual curve  $C-F$ to pass through $P \ominus e \in \sigma_{\infty}$ and $P \+ e \in \sigma_0$.

 Let $l \geq 1$ be an integer. A sequence $\{F_1, \ldots, F_l\}$ of $l$ distinct fibers is called a {\it special $l$-sequence} if
\[ F_i \cap \sigma_{\infty} = (F_{i-1} \cap \sigma_{\infty}) \+ e \; \; \mbox{for all} \; \; i=2, \ldots,l.\]
A sequence $\{P_1, \ldots, P_l\}$ of  $l$ distinct points on 
$\sigma_{\infty} \cong E$ is called a {\it special $l$-sequence} if
\[ P_i = P_{i-1} \+ e \; \; \mbox{for all} \; \; i=2, \ldots,l;\]
in other words, the sequence is $\{P, P \+ e, P \+ e^{\+2}, \ldots, P \+ e^{\+(l-1)}\}$.
By definition, the union of the fibers in a special $l$-sequence of fibers intersects
$\sigma_{\infty}$ along a special $l$-sequence of points. Conversely, any such sequence of points uniquely determines a special sequence of fibers. The pair of points  $P_0 :=P \ominus e \in \sigma_\infty$ and  $P_{l+1}:=P \+ e^{\+ l} \in \sigma_0$ will be called {\it the pair of points 
associated  with the special $l$-sequence of fibers or points}. We consider $P_0 +P_{l+1}$ as a point in $\Sym^2(E)$; then, $P_0 +P_{l+1}$ naturally lies on the curve $ \mathfrak{c}_{e,l+1}:=\left\{x+(x\+e^{\+l+1})\in \Sym^2(E)\; |\, x\in E\right\}$. More precisely, in view of the above considerations, the following holds:

\begin{lemma} \label{lemma:coppie}
A special $l$-sequence of fibers determines 
and is determined by a point on $\mathfrak{c}_{e,l+1}$. 
\end{lemma}

Figure \ref{fig:abe1} below, where the arrows between two points indicate a shift by $\+ e$ on $\sigma_0$, shows a special $5$-sequence
of fibers $\{F_1, \ldots, F_5\}$ with corresponding special $5$-sequence
of points $\{P_1, \ldots, P_5\}$. The associated pair of points on $E \cong \sigma_{\infty}\cong\sigma_0$ is $\{P_0,P_6\}$. 
(Recall that the identification of the two sections $\sigma_{\infty}$ and  $\sigma_0$ with $E$ is made through the fibers $F$.)

\begin{figure}[ht] 
\[
\includegraphics[width=10cm]{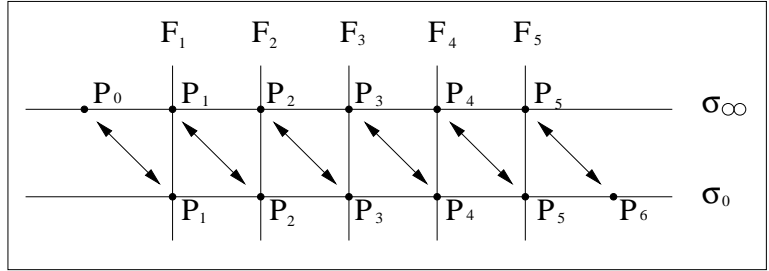}
\]
\caption{A special $5$-sequence
of fibers and points}
\label{fig:abe1}
\end{figure}

We remark that the unique irreducible component $\Gamma$ of a curve $C\in V_{|W|, \delta}(R)$ that is not a fiber is a section of $R$ over $E$ and 
$\Gamma \cap \sigma_{\infty}$ does not contain any {\it special $l$-sequence} of points, for any $l \geq 1$. We will now describe the intersection points of $\Gamma$ with $\sigma_\infty$ and $\sigma_0$.

As soon as $C$ contains a special $l$-sequence of fibers not contained in any special $(l+1)$-sequence, with associated pair $(P_0,P_{l+1})$, the curve $C$, and more precisely its irreducible component $\Gamma$, passes through the 
points $P_0 \in \sigma_{\infty}$ and $P_{l+1} \in \sigma_0$ by the gluing condition \eqref{eq:cond}. We call such a pair of points a {\it distinguished pair of points on $\Gamma$ of the first kind}. On the other hand, if a point $P\in \Gamma\cap \sigma_\infty$ is not associated to any special $l$-sequence of fibers contained in $C$ with $l\geq 1$, then $\Gamma$ must pass through  $P\oplus e\in \sigma_0$ and we call the pair $(P,P\+e)$ a {\it distinguished pair of points on $\Gamma$ of the second kind}.

Figure \ref{fig:abe2} below shows a curve $C$ in $|W|$ containing the special $5$-sequence
of fibers in Figure \ref{fig:abe1}, with irreducible component that is not a fiber being $\Gamma$.  Its image $\nu(C)$ in $S_0$ is also shown.
The distinguished pair of points on $\Gamma$ of the first kind is $(P_0,P_6)$.

\begin{figure}[ht] 
\[
\includegraphics[width=8cm]{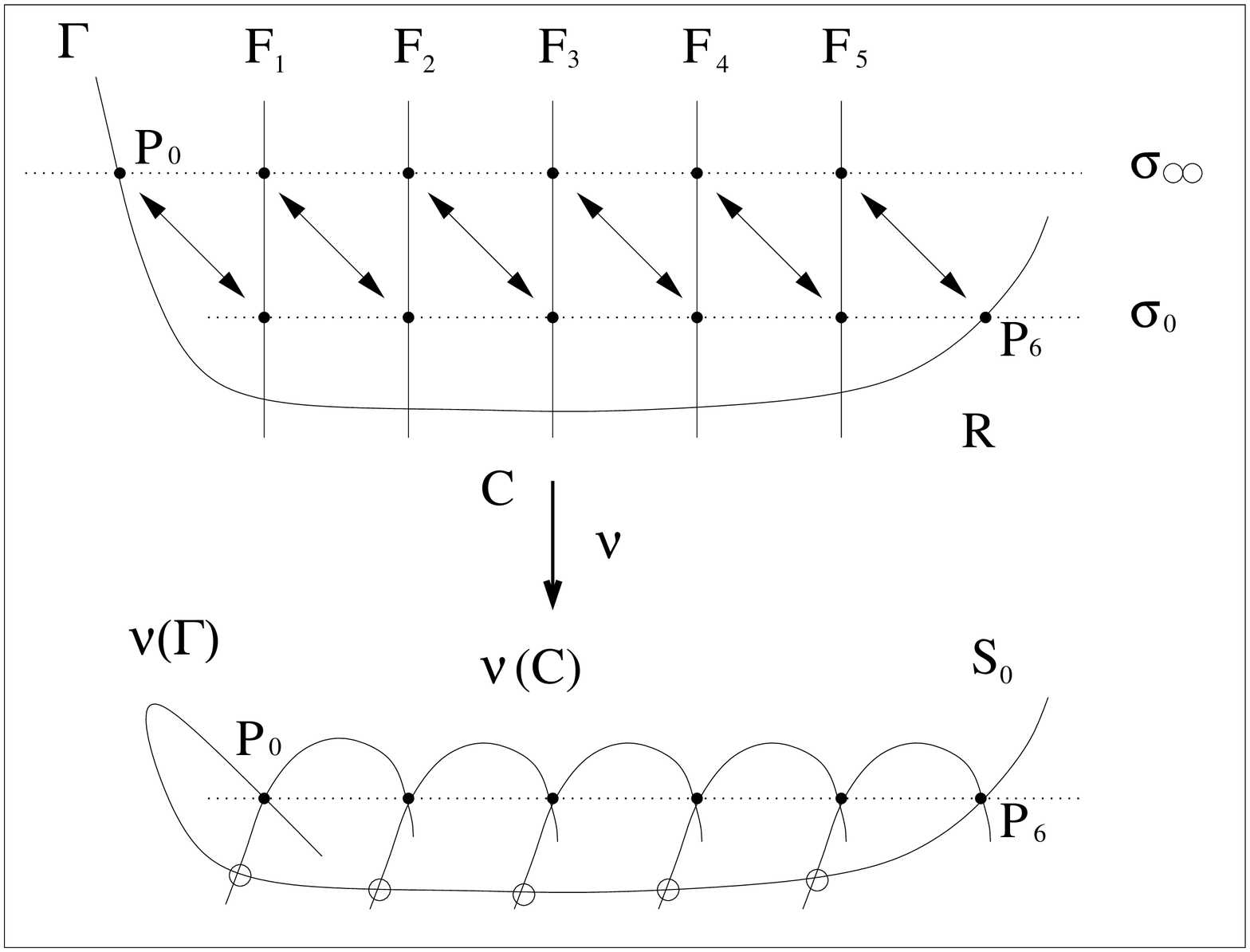}
\]
\caption{A curve in $|W|$ containing the special $5$-sequence
of fibers in Figure \ref{fig:abe1} and its image by $\nu$ in $S_0$.}
\label{fig:abe2}
\end{figure}

By the above argument, with any curve $C\in V_{|W|,\delta}(R)$ we can associate an $n$-tuple of nonnegative integers $(\alpha_0,\ldots, \alpha_{n-1})$ by setting
\begin{align*}
\alpha_0&:=\sharp\{\mbox{distinguished pairs of points on $\Gamma\subset C$ of the second kind}\},\\
\alpha_l&:=\sharp\{\textrm{special }l\textrm{-sequences of fibers in }C\textrm{ not contained in any }(l+1)\textrm{-sequence}\},\textrm{ for }l\geq 1.
\end{align*}
Since the union of fibers in each special $l$-sequence intersects $\sigma_{\infty}$ in $l$ points and $C$ also has to pass through the distinguished point $P_0 \in \sigma_{\infty}$, we have:
\begin{equation}
  \label{eq:numcond}
  n= \sharp\{C\cap\sigma_{\infty}\}= \sum_{j=0}^{ n-1} (j+1)\alpha_j.
\end{equation}
Furthermore, the number $\delta$ of fibers contained in $C$ is given by
\begin{equation}
  \label{eq:nodes}
\delta=\sum_{j=1}^{n-1} j\alpha_j,
\end{equation}
while
\begin{equation}
  \label{eq:distnumb}
\sum_{j=0}^{n-1} \alpha_j = \sharp\{\mbox{distinguished pairs of points on $\Gamma$}\; \mbox{(of first and second kind)}\}.
\end{equation}
\begin{remark}\label{zero}
With the above notation, the case $\delta=0$ corresponds to the $n$-tuple $(n,0,\ldots,0)$.
\end{remark}

For $l \geq 0$ and any point $P \in \sigma_{\infty}$,
we define the divisor
\[ D_l(P):=\sum_{i=0}^l (P \+ e^{\+i})=P+(P\+ e)+\cdots+(P\+e^{\+l}). \]
The following lemma is a straightforward consequence of the above discussion along with Lemma \ref{lemma:rest}, which implies that a curve $C\in\vert W\vert$ is uniquely determined by its intersection with $\sigma_\infty\simeq E$.  

\begin{lemma}\label{lemma:unico}
A curve $C$ in $V_{|W|,\delta}(R)$ determines and is completely determined by  an $n$-tuple of nonnegative integers $(\alpha_0,\ldots, \alpha_{n-1})$ fulfilling \eqref{eq:numcond} and \eqref{eq:nodes}, along with a set of $\alpha_l \leq n$  distinct points $\{P_{l,1}, \ldots, P_{l, \alpha_l}\}$
on $\sigma_{\infty}\simeq E$ for each $l=0, \ldots, n-1,$ such that the following is satisfied:
\begin{equation}
  \label{eq:condiper}
  \sum_{\stackrel{0 \leq l \leq n-1}{1 \leq j \leq \alpha_l}} D_l(P_{l,j}) \in |\O_E(\nu^*L_0)|,\,\, \textrm{ where } E\simeq \sigma_\infty.
\end{equation}
\end{lemma}

We denote by $\widetilde{V}(\alpha_0, \ldots, \alpha_{n-1})$ the subset of curves in $V_{|W|,\delta}(R)$ with associated vector $(\alpha_0, \ldots, \alpha_{n-1})$ satisfying conditions \eqref{eq:numcond} and \eqref{eq:nodes}. We define the locally closed subset of $|L_0|$:
$$
V(\alpha_0, \ldots, \alpha_{n-1}):=\{ X=\nu(C)\in |L_0|\;|\; C\in\widetilde{V}(\alpha_0, \ldots, \alpha_{n-1})\}.
$$
If $C=\Gamma\cup\bigcup_{i=1}^\delta F_i \in \widetilde{V}(\alpha_0, \ldots, \alpha_{n-1})$, the images under $\nu$ of the intersection points of $\Gamma$ with the $\delta$ fibers $F_i$ are the {\it marked nodes} of 
$X=\nu(C)$ as a curve in $V_{|L_0|,\delta}(S_0)$. In Figure \ref{fig:abe2} the
marked nodes of $X$, coming from the one specific special $5$-sequence of fibers depicted, are circled.

We set $p:=n+1$ and $g:=p-\delta$. Note that $p$ is the arithmetic genus of all curves in $\vert L_0\vert$. From \eqref{eq:numcond}-\eqref{eq:distnumb}, we have
\begin{equation}
  \label{eq:rela}
  g=1+\sharp\{\mbox{distinguished pairs of points on $\Gamma\subset C$}\}.
\end{equation}

\begin{lemma} \label{lemma:v0}
  Under conditions \eqref{eq:numcond} and \eqref{eq:nodes}, the following hold:
  \begin{itemize}
  \item[(i)] $V(\alpha_0, \ldots, \alpha_{n-1})$ fills up one or more components of $V_{|L_0|,\delta}(S_0)$ of the expected dimension $n-1-\delta=g-2$;
 \item[(ii)] $V_{\{L_0\},\delta}(S_0)$ is regular (i.e., smooth of the expected dimension $g$)  at any point of $V(\alpha_0, \ldots, \alpha_{n-1})$.
  \end{itemize}
\end{lemma}

\begin{remark}\label{cata}
In the statement of the  lemma we are implicitly using that $\dim\{L_0\}=\dim |L_0|+2$. We refer to \cite[II, 5C]{HKW} for an explanation of the two extra dimensions.
\end{remark}

\renewcommand{\proofname}{Proof of Lemma \ref{lemma:v0}}

\begin{proof}
Given a member $C \in \widetilde{V}(\alpha_0, \ldots, \alpha_{n-1})$,   we denote, for $l=0,\ldots ,n-1$, as in Lemma \ref{lemma:unico}, by 
$\{P_{l,j}\}_{1\leq j\leq \alpha_j}$ the points on $\sigma_{\infty} \cong E$ in the distinguished pairs on $\Gamma\subset C$. Then we have injective maps
 \[
\xymatrix{ 
V(\alpha_0, \ldots, \alpha_{n-1}) \ar[r]^{\hspace{-1.3cm}f} 
&  \Sym^{\alpha_0}(E) \x \cdots \x \Sym^{\alpha_{n-1}}(E)  \ar[r]^{\hspace{1.6cm}h} &  \Sym^n(E) 
\\
X=\nu(C)      \ar@{|->}[r] & \{P_{l,j}\} \ar@{|->}[r] & \sum D_l(P_{l,j}).
}
\]
The target space of $f$ is irreducible of dimension $\sum \alpha_l=n-\delta=\dim |W|+1-\delta$, and the image of $f$ equals $h^{-1}|\O_E(\nu^*L_0)|$ by \eqref{eq:condiper}. We make the following:

\begin{claim}\label{claim:utile}
 $h^{-1}|\O_E(\nu^*L_0)|$ is nonempty and different from $\Sym^{\alpha_0}(E) \x \cdots \x \Sym^{\alpha_{n-1}}(E)$.
 \end{claim}

Granting this, we can conclude that the image of $f$ is nonempty and equidimensional of codimension $1$ in $\Sym^{\alpha_0}(E) \x \cdots \x \Sym^{\alpha_{n-1}}(E)$. Therefore, also $V(\alpha_0, \ldots, \alpha_{n-1})$ is nonempty and part (ii) is a direct consequence of Proposition \ref{prop:rego}. Furthermore, since any curve in $V_{|L_0|,\delta}(S_0)$ lies in
$V(\beta_0, \ldots, \beta_{n-1})$ for some $\beta_i$, it follows that $V(\alpha_0, \ldots, \alpha_{n-1})$ is the union of some irreducible components of $V_{|L_0|,\delta}(S_0)$ of dimension $n-1-\delta=g-2$, proving (i). 

It remains to verify Claim \ref{claim:utile}. It is enough to show that, given any point $D:=\sum  D_l(P_{l,j})$ in the image of $h$ and any other divisor $D'$ of degree $n$ on $E$, we can find a point $P\in E$ such that  
\begin{equation}\label{equation:claim}
h(\{P_{l,j}\+P\})=\sum  D_l(P_{l,j})\+ P=D\+ P\in |D'|.
\end{equation}
We use the isomorphism $\Pic^n(E)\simeq E$ and the fact that $ne_0\sim nQ$ if and only if $Q$ has $n$-torsion (indeed, $nQ\sim Q^{\+ n}+(n-1)e_0$ for any $Q\in E$). Therefore, any complete linear system of degree $n$ on $E$ contains elements supported at only one point. In particular, we find a point $P\in E$ such that $$nP\sim D'+ne_0-D,$$ or equivalently by (\ref{equation:law}), condition (\ref{equation:claim}) is satisfied. This concludes the proof.
\end{proof}

\renewcommand{\proofname}{Proof}

\begin{remark}
The proof of Claim \ref{claim:utile} also shows that $h^{-1}|A|\simeq h^{-1}|\O_E(\nu^*L_0)|$ for all $A\in\Pic^n(E)$.
\end{remark}

Using \eqref{eq:rela}, we obtain the following result:

\begin{lemma} \label{lemma:stablemodel}
Let $X=\nu(C)$ be a curve in ${V}(\alpha_0, \ldots, \alpha_{n-1})$, where $C=\Gamma\cup\bigcup_{i=1}^\delta F_i$. Then $\Gamma$ is a section of $R$ over $E$. Furthermore, the stable model $\overline{X}$ of the partial normalization of $X$ at its $\delta$ marked nodes is the image of $\Gamma \cong E$ under the morphism identifying each distinguished pair of points. In particular, $\overline{X}$ has arithmetic genus $g$.
\end{lemma}

In Figure \ref{fig:abe3} below we show the curve $X=\nu(C)$ from Figure \ref{fig:abe2}, which has $\alpha_5>0$. The normalization separates the images of the fibers $F_1, \ldots, F_6$ from $\nu(\Gamma)$, except for the intersections at the distinguished pair of points $P_0$ and $P_6$. In the stable equivalence class, all images of the fibers are contracted, and the points $P_0$ and $P_6$ are identified.

\begin{figure}[ht] 
\[
\includegraphics[width=7cm]{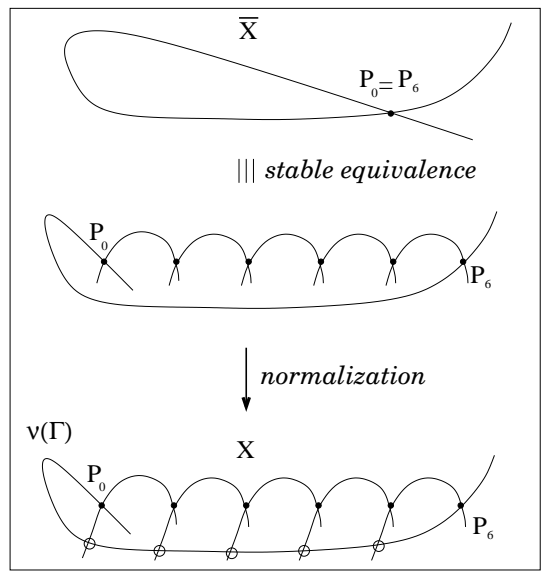}
\]
\caption{The curve $X=\nu(C)$ from Figure \ref{fig:abe2}, its normalization and stable model.}
\label{fig:abe3}
\end{figure}

We return to the family $f: \mathcal{S} \to \mathbb{D}$  in Definition \ref{def:dege} and Proposition \ref{frittata},  with $e\in E$ nontorsion. Let $\H$ be the component of the relative Hilbert scheme containing $\PP(f_*L)$  and $\V_{\{\L\},\delta}^* \to \DD^*=\DD \setminus \{0\}$ the {\it relative Severi variety} of $\delta$-nodal curves in $\H$, with fiber over $t\in  \DD^*$ equal to
$V_{\{L_t\},\delta}(S_t)$. 

\begin{lemma}\label{lem:defo} 
 Assume that $V_{|L_0|, \delta}(S_0) \neq \emptyset$ and let
 $X_0$
be a point of $V_{|L_0|, \delta}(S_0)$.
Then $X_0$ sits in the closure of
$\V_{\{\L\},\delta}^*$  in $\H$ and  $\V_{\{\L\},\delta}^*$   dominates $\mathbb D$.
In particular, for general $t \in \mathbb D$, the Severi variety $V_{\{L_t\},\delta}(S_t)$ is nonempty and regular. 
\end{lemma}

\begin{proof} 
 This is proved as in \cite[Lemma 1.4]{CK}; regularity is immediate from  Proposition \ref{prop:rego}.
\end{proof}

Now we can prove the first main result of the paper.

\renewcommand{\proofname}{Proof of Theorem \ref{thm:main1}}

\begin{proof}
  By Lemma \ref{lemma:v0} there is a  component of $V_{|L_0|,\delta}(S_0)$ obtained by taking any $n$-tuple of nonnegative integers $(\alpha_0,\ldots,\alpha_{n-1})$ that fulfills \eqref{eq:numcond} and \eqref{eq:nodes}. For instance, take $\alpha_0=n-\delta-1$, $\alpha_\delta=1$ and $\alpha_i=0$ for $i\neq 0,\delta$. Hence, the statement follows from Lemma \ref{lem:defo} and Remark \ref{rem:sticazzi}.
\end{proof}
\renewcommand{\proofname}{Proof}

Lemma \ref{lem:defo} also yields the existence of a partial compactification of $\V_{\{\L\},\delta}^*$ containing the curves in $V_{\{L_0\},\delta}(S_0)$:

\begin{cor} \label{cor:esisterel}
   Let $f:\SS \to \DD$ be as in Definition \ref{def:dege} and Proposition \ref{frittata}. Then there is an $f$-relative Severi variety $\phi_{\{\L\},\delta}: \V_{\{\L\},\delta} \to \DD$ with fibers $V_{\{L_t\},\delta}(S_t)$ for all $t \in \DD$, such that 
$\V_{\{\L\},\delta}$ is smooth of pure dimension $n+2-\delta$. Moreover, each component of $\V_{\{\L\},\delta}$ dominates $\DD$.
\end{cor}

\begin{proof}
  The same argument as in \cite[Lemma 1.4]{CK} proves that every component of a Severi variety on a single surface belongs to a component of $\V_{\{\L\},\delta}
$ of dimension one more, thus dominating $\DD$. (We remark that the  nodes of a curve lying on the singular locus of $S_0$ smooth when $S_0$ smooths, cf., e.g., \cite{Ch,Ga}). Smoothness of $\V_{\{\L\},\delta}$ follows since all fibers are smooth by Proposition \ref{prop:rego}. 
 \end{proof}

\section{Variation of curves in moduli} \label{sec:modular}

In this section we will prove Theorem \ref{thm:moduli}.

For $n\geq 3$, let $A(1,n)$ be the fine moduli space of abelian surfaces with polarization of type $(1,n)$ and level structure of canonical type; for $n=2$, we define $A(1,2)$ as the moduli  space of abelian surfaces with polarization of type $(1,2)$ and a suitable level structure that makes it a fine moduli space. For $n\geq 2$, we consider the universal family of abelian surfaces $\phi_n:\mathcal{S}_n\to A(1,n)$. There exists a line bundle $\L_n$ on $\mathcal{S}_n$ whose restriction to the fiber  of $\phi_n$ over a point $(S,L,\alpha) \in A(1,n)$ is the $(1,n)$-polarization $L$ (and $\alpha$ is a level structure). For any $0\leq \delta\leq n-1$, let $\V_{n,\delta}$ be the universal Severi variety along with the projection $\pi_{n,\delta}: 
\V_{n,\delta} \to A(1,n)$ with fibers $V_{\{L\},\delta}(S)$. As in the proof of Corollary \ref{cor:esisterel}, one shows that the scheme
$\V_{n,\delta}$ is equidimensional of dimension $n+4-\delta= g+3$ (possibly with more components). We have a moduli map
\[
\xymatrix{
 \psi_{n,\delta}: \V_{n,\delta} \ar[r] & \M_g
}\]
mapping a curve to the class of its normalization and we set $\mathfrak{A}_{g,n}: =\im \psi_{n,\delta}$, as in the introduction. The fibers of $\psi_{n,\delta}$ restricted to any component of $\V_{n,\delta}$ are at least two-dimensional, due to the possibility of moving a curve by translation on each single surface. In fact, Theorem \ref{thm:moduli} will follow if we prove that a general fiber of $\psi_{n,\delta}$ restricted to some component of $\V_{n,\delta}$ has dimension $2$. 

First of all, we remark that for $g=2$ the map $\psi_{n,n-1}$ restricted to any component of $\V_{n,n-1}$ has two-dimensional fibers (and is thus dominant). This can be verified as follows. Any $[C]\in V_{\{L\},n-1}(S)$, with $(S,L)$ a $(1,n)$-polarized abelian surface, provides an isogeny $J(\widetilde{C})\to S$ of  bounded degree, where $J(\widetilde{C})$ is the Jacobian of the normalization $\widetilde{C}$ of $C$. It is then enough to use the finiteness of the number of isogenies of bounded degree from  $J(\widetilde{C})$ to some $(1,n)$-polarized abelian surfaces, along with the fact that any component of $V_{\{L\},n-1}(S)$ is two-dimensional by Proposition \ref{prop:rego}. 

By standard deformation theory, the regularity of $V_{\{L\},\delta}(S)$ implies that the nodes of any curve $[C] \in V_{\{L\},\delta}(S)$ can be smoothed independently; in particular, for any $0 \leq \delta' \leq \delta$,  the Severi variety $V_{\{L\},\delta}(S)$ is contained in the Zariski closure of $V_{\{L\},\delta'}(S)$ in $\{L\}$. As a consequence, there is a partial compactification of $\V_{n,\delta}$:
\begin{equation} \label{eq:partcomp} 
\overline{\V}_{n,\delta} := \cup_{j=\delta}^{n-1} \V_{n,j} .
\end{equation}
. 

Pick any component $\V^\dagger_{n,n-1}$ of $\V_{n,n-1}$. Then, for each $0 \leq \delta < n-1$, choose
components $\V_{n,\delta}^\dagger$ of $\V_{n,\delta}$ such that
\[ \V^\dagger_{n,n-1} \subset \overline{\V}^\dagger_{n,n-2} \subset \cdots \subset \overline{\V}^\dagger_{n,1} \subset \overline{\V}^\dagger_{n,0}, \]
where we let $\overline{\V}_{n,\delta}^\dagger$ be the partial compactification of 
$\V_{n,\delta}^\dagger$ induced by \eqref{eq:partcomp}. The next result proves Theorem \ref{thm:moduli}.

\begin{thm} \label{thm:modcasogen}
  A general fiber of the map $(\psi_{n,\delta})_{|\V_{n,\delta}^\dagger}$ is  two-dimensional. 
\end{thm}

\begin{proof}
As in \cite{Tan,CFGK} one can prove the existence of a scheme
\[ \mathcal{W}_{n,\delta}:=\Big\{ (C,N) \; | \; C \in  \V_{n,\delta'}^\dagger \; \mbox{for some $\delta' \geq \delta$ and $N$ is a subset of $\delta$ of the nodes of $C$} \Big\}. \]
The scheme $\V_{n,\delta}^\dagger$ can be identified with a dense open subset of 
$\mathcal{W}_{n,\delta}$. We also have an extended moduli map $\tilde{\psi}_{n,\delta}: \mathcal{W}_{n,\delta} \to \overline{\M}_g$, 
mapping a pair $(C,N)$ to the class of the partial normalization of $C$ at $N$ (where $\overline{\M}_g$ is the Deligne-Mumford compactification of $\M_g$).
The result will follow if we prove that a general fiber of $ \tilde{\psi}_{n,\delta}$ is at most two-dimensional. Take a general curve $C \in \V_{n,n-1}^\dagger$
and choose a subset $N$ of $\delta$ of its nodes. Then $(C,N) \in \mathcal{W}_{n,\delta}$. By the result in genus two, the fiber over $\tilde{\psi}_{n,\delta}((C,N))$ is (at most) two-dimensional, and the result follows by semicontinuity.
\end{proof}

We remark that, as in the $K3$ case, the following very interesting questions are still open: 

\begin{question} For a general  $(1,n)$-polarized abelian surface $(S,L)$, is the Severi variety $V_{\{L\},\delta}(S)$ irreducible? Is the universal Severi variety $\V_{n,\delta}$ irreducible? 
\end{question}

\section{Linear series on curves on abelian surfaces} \label{sec:lser}

\subsection{The Brill-Noether loci} 
Given a surface $S$ (possibly having normal crossing singularities), the moduli morphism
\begin{equation}\label{eq:uno}
\xymatrix{\psi_{S, \{L\},\delta} :V_{\{L\}, \delta}(S) \ar[r] &
\overline{\M}_g}
\end{equation}
assigns to a curve $C\in V_{\{L\},\delta}(S)$ the isomorphism class of the stable model $\overline C$
of its partial normalization $\widetilde C$ at its $\delta$ marked nodes.  We sometimes simply write $\psi$ to ease notation.

Having fixed two integers $r\geq 1$ and $d\geq 2$, one defines the Brill-Noether locus
$$\M^r_{g,d}:=\{[C]\in\M_g\,\vert\,G^r_d(C)\neq\emptyset\}.$$
This coincides with $\M_g$ if and only if the Brill-Noether number $\rho(g,r,d):=g-(r+1)(g-d+r)$ is nonnegative; if $\rho(g,r,d)<0$, then the codimension of $\M^r_{g,d}$ inside $\M_g$ is at most $-\rho(g,r,d)$ by \cite{St}.

If $(S,L)$ is a (possibly degenerate) $(1,n)$-polarized abelian surface, we define the scheme
$$
\{L\}^r_{\delta,d}:=\{C \in V_{\{L\}, \delta}(S) \; | \;
\psi(C)\in \overline{\M}^r_{g,d}\Big \},
$$
where  $\overline{\M}^r_{g,d}$ is the Zariski closure of $\M^r_{g,d}$ in $\overline{\M}_g$. We also set
 $$|L|^r_{\delta,d}:=\{L\}^r_{\delta,d}\cap |L|.$$
 When $\delta=0$, we simplify notation and denote by $|L|^r_{d}$ and $\{L\}^r_d$ the Brill-Noether loci of smooth curves in $|L|$ and $\{L\}$, respectively. 

\begin{proposition} \label{prop:expdimvk}
Let $(S,L)$ be a possibly degenerate $(1,n)$-polarized abelian surface. The \linebreak  expected dimension of every irreducible component of $|L|^r_{\delta,d}$ (respectively $\{L\}^r_{\delta,d}$) equals \linebreak $\min\{g-2,g-2+\rho(g,r,d)\}$ (resp. $\min\{g,g+\rho(g,r,d)\}$). 
\end{proposition}

\begin{proof}
 Consider the moduli morphism $\psi$ in \eqref {eq:uno}.
Let $Z$ be an irreducible component of $|L|^r_{\delta,d}$ such that $\psi(Z)$ is a component of $\M\cap \psi(V)$, where $\M$ and $V$ are irreducible components of $\overline{\M}^r_{g,d}$ and $V_{\{L\},\delta}(S)$, respectively. If $s$ is the dimension of a general fiber of $\psi|_V$, then
\begin{align}\label{muffin}
\dim Z\geq \dim\psi(Z)+s\geq \dim\psi(V)-\codim_{\overline{\M}_g} \M+s & =\dim V-\codim_{\overline{\M}_g} \M\\\nonumber &\geq\dim V+ \min\{\rho(g,r,d), 0\}.
\end{align}
 The statement for $|L|^r_{\delta,d}$ then follows because $\dim  V=g-2$. The proof for $\{L\}^r_{\delta,d}$ is very similar. 
\end{proof}

\begin{remark}\label{pencil}
When $r=1$, the scheme $|L|^1_{\delta,k}$ is called the {\em $k$-gonal locus}. This  case is quite special. For instance, if $\rho(g,1,k)<0$, the Brill-Noether locus $\M^1_{g,k}$ is known to be irreducible of the expected dimension $2g+2k-5$. Furthermore, a curve $C$ lies in $|L|^1_{\delta,k}$ if and only if the partial normalization $\widetilde C$ of $C$ 
at its $\delta$ marked nodes is stably equivalent to a curve that is the domain of an admissible cover of degree $k$ to a stable pointed curve of genus $0$  \cite[Thm. 3.160]{HM}.
\end{remark}

When $r=1$ and $(S,L)$ is general, we can  bound the dimension of the $k$-gonal locus from above, even for curves with arbitrary singularities. 

\begin{thm} \label{thm:boundingdim}
  Let $(S,L)$ be a polarized abelian surface such that $[L] \in \NS(S)$ has no decomposition into nontrivial effective classes. Assume that $V \sub \{L\}$ is a nonempty reduced scheme parametrizing a flat family of irreducible curves of geometric genus $g$ whose normalizations possess linear series of type $g^1_k$. Let $\widetilde{C}$ be the normalization of a general curve parametrized by the family. Then
\begin{equation}
  \label{eq:bounddim1}
 \dim V + \dim G^1_k(\widetilde{C}) \leq 2k-2. 
\end{equation}

In particular,
\begin{equation}
  \label{eq:bounddim2}
 \dim V \leq \min \{g,2k-2\}. 
\end{equation}
\end{thm}

\begin{proof}
 Possibly after shrinking $V$ and diminishing $k$, we may without loss of generality assume that all linear series of type $g^1_k$ in question are base point free. 
Let $\C \to V$ be the universal family. Normalizing $\C$, and possibly after restricting to an open dense subscheme of $V$, we obtain a flat family $\widetilde{\C} \to V$ of smooth, irreducible curves possessing linear series of type $g^1_k$ \cite[Thm. 1.3.2]{tes} and a natural morphism 
$\Sym^k(\widetilde{\C}/V) \to \Sym^k(S)$, where $\Sym^k(\widetilde{\C}/V)$ is the relative symmetric product. Any $g^1_k$ on the normalization of a curve in the family defines a rational curve inside $\Sym^k(\widetilde{\C}/V)$.  All such curves are mapped to distinct, rational curves in $ \Sym^k(S)$  (cf. e.g. \cite{CK});  moreover, as the linear series are base point free, none of the curves are  contained in the singular locus $\Sing(\Sym^k(S))$.  Via the Hilbert-Chow morphism
$\mu: \Hilb^k(S) \to \Sym^k(S)$ we obtain a flat family of rational curves in $\Hilb^k(S)$ of dimension $\dim V+\dim G^1_k(\widetilde{C})$, with $\widetilde{C}$ as in the statement. Any rational curve is contracted by the composed morphism
$\Sigma \circ \mu$, where $\Sigma:\Sym^k(S) \to S$ is the summation map, whence it lies in some fiber $K_x^{[k-1]}(S):=(\Sigma \circ \mu)^{-1}(x) \subset \Hilb^k(S)$, $x \in S$, which is well-known to be a smooth hyperk{\"a}hler manifold of dimension $2(k-1)$ (all $K_x^{[k-1]}(S)$ are isomorphic). At the same time, having fixed a curve $[C]\in V$ and a $g^1_k$ on its normalization, one obtains a rational curve in any fiber of $\Sigma \circ \mu$ by suitably translating $C$ in $S$. Hence we get a flat family of rational curves in $K^{[k-1]}(S):=K_{e_0}^{[k-1]}(S)$  of dimension $\dim V-2+\dim G^1_k(\widetilde{C})$. 

We borrow an argument from \cite{AV}.
Let $X \subset K^{[k-1]}(S)$ be the closure of the locus covered by the rational curves in our family and  $a$ its codimension. By \cite[Cor. 5.2]{ran}, the dimension of our family is at least $\dim K^{[k-1]}(S)-2=2k-4$. Hence, through a general point $\xi \in X$, there is a family of rational curves of dimension at least $a-1$. 
Denote by $X_{\xi} \subset X$ the closure of the locus they cover. 

\begin{claim}\label{prosecco}
 Only finitely many rational curves of the family pass through $\xi$ and a general $\xi'\in X_{\xi}$.  
\end{claim}

Granting this, we have that $\dim X_{\xi} \geq a$. We can now proceed as in the last part of the proof of \cite[Thm. 4.4]{AV} and conclude that $\dim X_{\xi}=a$. A posteriori, the dimension of the family of curves through $\xi$ is precisely $a-1$ and that of the total family is $2k-4$. It follows that $\dim V-2+\dim G^1_k(\widetilde{C}) \leq 2k-4$, proving \eqref{eq:bounddim1}. Inequality \eqref{eq:bounddim2} is immediate, since $\dim V \leq g$ by \cite[Prop. 4.16]{DS}.

We are left with proving Claim \ref{prosecco}. To this end, consider the incidence scheme 
\[ I:= \{ (P,Z) \; | \; P \in \Supp Z\} \subset S \x K^{[k-1]}(S), \]
with its two projections
\[
\xymatrix{
I \ar[r]^{\hspace{-0.7cm}\alpha} \ar[d]_{\beta} & K^{[k-1]}(S) \\
S.
}
\]
The map $\alpha$ is finite of degree $k$.  Let $R \subset K^{[k-1]}(S)$ denote a rational curve determined by a $g^1_k$ on  $\widetilde{C}$. The map $\beta$ is generically one-to-one on $\alpha^{-1}(R)$ and $\beta(\alpha^{-1}(R))=
\beta_*(\alpha^{-1}(R))=C$.

If the claim were false, then by Mori's bend-and-break technique $R$ would be algebraically equivalent to a curve $R_0  \subset X_{\xi}$ having either a nonreduced component $R^{(0)}_0$ passing through both $\xi$ and $\xi'$ or two distinct components $R^{(1)}_0$ and $R^{(2)}_0$ passing through $\xi$ and $\xi'$, respectively (see \cite[Pf. of Lemma 1.9]{komo}). Since $\xi, \xi' \not \in \Exc \mu \cap  K^{[k-1]}(S)$, the support of the zero schemes parametrized by $R^{(j)}_0$, for $j=0,1,2$, spans a curve on $S$, so that $\beta(\alpha^{-1}(R^{(j)}_0))$ is a curve on $S$. Hence $\beta_*(\alpha^{-1}(R_0)) \equiv C \in \{L\}$ contains either a nonreduced component or two distinct components, a contradiction.
\end{proof}

 Using the fact that the condition on $L$ is open in the moduli space of polarized abelian surfaces, and combining with Proposition \ref{prop:expdimvk}, we obtain the following:

\begin{cor} \label{cor:tuttecompdim}
   Let $(S,L)$ be a general $(1,n)$-polarized abelian surface. If $\{L\}^1_{\delta,k} \neq \emptyset$, then each component has dimension $\min \{g,2k-2\}$ and the normalization $\widetilde{C}$ of a general curve therein satisfies $\dim G^1_k(\widetilde{C})=\max\{0,\rho(g,1,k)\}$. 
 \end{cor}

The above corollary proves Theorem \ref{thm:main2}(ii). In the next section, we will provide an independent proof of (ii) for {\it one} component of $|L|^1_{\delta,k}$ by degeneration, which also yields part (iii) of the theorem.
As a further consequence, we can prove Theorem 
\ref{thm:conjDS}.

\renewcommand{\proofname}{Proof of Theorem \ref{thm:conjDS}}

\begin{proof}
  Let $V_{g} \sub \{L\}$ be any component of the locus parametrizing irreducible curves of geometric genus $g$. Then $\dim V_{g}=g$ by \cite[Prop. 4.16]{DS}.
Therefore, by Theorem \ref{thm:boundingdim}, a general curve parametrized by 
$V_{g}$ has nontrigonal normalization as soon as $g\geq 5$. By  \cite[Prop. 4.16]{DS}, it must therefore be nodal.
\end{proof}

\renewcommand{\proofname}{Proof}

\begin{remark} \label{rem:ancheK3}
Theorem \ref{thm:boundingdim} stills holds with basically the same proof if one replaces $(S,L)$ with a general primitively polarized $K3$ surface. As a consequence, the same applies to Corollary \ref{cor:tuttecompdim} and Theorem \ref{thm:conjDS}; in particular, the former implies that \cite[Thm 0.1(ii)]{CK} holds on any component of the $k$-gonal locus, and the latter improves \cite[Prop. 4.10]{DS}.

\end{remark}

\subsection{Necessary conditions for nonemptiness of Brill-Noether loci} \label{ssec:nec}

 The next result proves Theorem \ref{thm:main4}(iii), as the condition on $[C]\in \NS(S)$ is open in the moduli space of  abelian surfaces.

\begin{thm} \label{thm:necex}
  Let $C$ be a reduced and irreducible curve of arithmetic genus $p=p_a(C)$ on an abelian surface $S$ such that $[C]\in \NS(S)$ has no decomposition into nontrivial effective classes. Assume that $C$ possesses a torsion free rank one sheaf $\A$ such that $\deg \A=d$ and $h^0(\A)=r+1$. Then one has:
  \begin{equation}
    \label{eq:rhocond}
   \rho(p,r,d) \geq -r(r+2), 
  \end{equation}
or equivalently,
\begin{equation}\label{liscio}
(r+1)d\geq r(p-1).
\end{equation}
\end{thm}

In the case where $C$ is a smooth curve and $\A$ is a line bundle, this result is one of the two statements in the main theorem of \cite{Pa}. The proof we present here is slightly simpler.
We need the following variant of standard results (cf. \cite{La}). 

\begin{lemma} \label{lemma:fact}
  Let $\F$ be a sheaf of positive rank on a smooth projective surface $X$ that is generically generated by global sections and such that $h^2(\F \* \omega_X)=0$. Then: 
\begin{itemize}
\item[(i)] $c_1(\F)$ is represented by an effective nonzero divisor;
\item[(ii)] if $\F$ is nonsimple (i.e. it has nontrivial endomorphisms), then 
$c_1(\F)$ can be represented by a sum of two effective nonzero divisors.
\end{itemize}
\end{lemma}

\begin{proof}
 We first prove (i).  Let $T(\F)$ be the torsion subsheaf of $\F$ and set $\F_0:=\F/T(\F)$. Then $c_1(\F)=c_1(\F_0) + c_1(T(\F))$
and $c_1(T(\F))$ is represented by a nonnegative linear combination of the codimension one irreducible components of the support of $T(\F)$, if any. Moreover, $\F_0$ is torsion free of positive rank and, being a quotient of $\F$, is globally generated off a codimension one set and satisfies $h^2(\F_0\* \omega_X)=0$. We may therefore assume that $\F$ is torsion free. 

Let $\F'$ denote the image of the evaluation map $H^0(\F) \* \O_X \to \F$, which is generically surjective by assumption. We have $c_1(\F)=c_1(\F')+D$,
where $D$ is effective and is zero if and only if the evaluation map is surjective off a finite set.
The sheaf $\F'$ is torsion free of positive rank and generated by its global sections. By a standard argument as in \cite[p.~302]{La}, $c_1(\F')$ is represented by an effective divisor, which is zero if and only if $\F'$ is trivial. We are now done, unless $\F'$ is trivial and $D$ is zero, which means that we have an injection $\F' \cong \+ \O_X \to \F$ with cokernel supported on a finite set. But this implies $h^2(\F \* \omega_X)=h^2(\+ \omega_X) >0$, a contradiction. 

 We next prove (ii). If  $\F$ is nonsimple, then by standard arguments as in e.g. \cite{La}, there exists a nonzero endomorphism $\nu: \F \to \F$ dropping rank everywhere. The sheaves $\N:=\im \nu$ and $\M:=\coker \nu$ have positive ranks, and, being quotients of $\F$, are generically globally generated and satisfy 
$H^2(\N\* \omega_X)=H^2(\M\* \omega_X)=H^2(\F \* \omega_X)=0$. Hence, by (i), both $c_1(\N)$ and $c_1(\M)$ are represented by an effective nonzero divisor. Since $c_1(\F)=c_1(\N)+c_1(\M)$, the result follows. 
\end{proof}

\renewcommand{\proofname}{Proof of Theorem \ref{thm:necex}}

\begin{proof}
  Let $\A'$ denote the globally generated part of $\A$, that is, the image of the evaluation map $H^0(\A) \* \O_C \rightarrow \A$. Then $h^0(\A')=h^0(\A)$ and $\deg \A' \leq \deg \A$, whence it suffices to prove the statement for $\A'$. We may thus assume that $\A$ itself is globally generated.

We consider $\A$ as a torsion sheaf on $S$. By arguing as in \cite{Go}, the kernel of the evaluation map is a vector bundle, whose dual we denote by $\E_{C,\A}$. This gives the exact sequence
  \begin{equation}
    \label{eq:defE}
   \xymatrix{
0 \ar[r] &  \E_{C,\A}^\vee \ar[r] &  H^0(\A) \* \O_S \ar[r]^{\hspace{0.9cm}ev_{S,\A}} & \A  \ar[r] & 0.} 
  \end{equation}
This sequence and its dual trivially yield:
\begin{eqnarray}
\label{eq:p0}  \rk \E_{C,\A}=r+1; \; \; c_1(\E_{C,\A})=[C], \; \; c_2(\E_{C,\A})=d; \; \; \chi(\E_{C,\A})=p-d-1; \\
\label{eq:p3} h^0(\E_{C,\A}^\vee)=0  \; \; \mbox{and} \; \; h^0(\E_{C,\A}) \geq r+1;\\
\label{eq:p4} \mbox{$\E_{C,\A}$ is globally generated off $C$}.  
 \end{eqnarray}
By \eqref{eq:p3} and \eqref{eq:p4}, along with Lemma \ref{lemma:fact}(ii) and our assumptions on $[C]$, we have that $\E_{C,\A}$ is simple. A simple vector bundle $\E$ on an abelian surface must satisfy 
\begin{equation} \label{eq:simple}
\frac{1}{2}c_1(\E)^2 \geq \rk \E \cdot \chi(\E).
\end{equation}
(Indeed, by Serre duality, $1=h^0(\E \* \E^*)= h^2(\E \* \E^*)=\frac{1}{2} \left( \chi (\E \* \E^*)+h^1(\E \* \E^*) \right)$; moreover, $h^1(\E \* \E^*) \geq 2$, 
cf. e.g. \cite[Pf. of 10.4.4, p.~296]{BL}, whence \eqref{eq:simple} follows by Riemann-Roch.) The desired inequality \eqref{eq:rhocond} now follows from \eqref{eq:p0}.
\end{proof}

\renewcommand{\proofname}{Proof}

As a consequence, we obtain a result in the spirit of \cite[Thm.~3.1]{CK}.

\begin{thm} \label{thm:noexist}
 Let $C$ be an integral curve of arithmetic genus $p=p_a(C)$ on an abelian surface $S$ such that $[C] \in \NS(S)$ has no decomposition into nontrivial effective classes. Assume that the normalization of $C$ possesses a $g^r_d$. Let $g$ be the geometric genus of $C$ and set $\delta=p-g$. Then
  \begin{equation}
    \label{eq:bound}
  \rho(p,\alpha r,\alpha d+\delta)  \geq -\alpha r(\alpha r+2), \; \; \mbox{i.e.,} \; \; \delta \geq \alpha\Big(r(g-d\alpha-1)-d\Big),
      \end{equation}
where $\alpha=\alpha(p,r,d,\delta):=\Big\lfloor \frac{r(g-1)+d(r-1)}{2rd}\Big\rfloor$.  
\end{thm}

\begin{proof}
  Let $\nu: \widetilde{C} \to C$ be the normalization of $C$. There exists a line bundle $A\in\Pic^d(\widetilde{C})$ such that $h^0(\widetilde{C},A)\geq r+1$. Then, for any positive integer $l$, the sheaf $\A_l:=\nu_*(A^{\* l})$ is torsion free of rank one on $C$ with $h^0(\A_l)=h^0(A^{\*l})\geq lr+1$ and $\deg \A_l=\deg (A^{\*l})+\delta=ld+\delta$.
  
We have $\rho(\A_l) \leq \rho(p,lr,ld+\delta)$.
Theorem \ref{thm:necex} then yields
\begin{equation}
  \label{eq:rho0}
  \rho(p,lr,ld+\delta)=l^2 r(d-r)-l(gr+r-d)+\delta   \geq -lr(lr+2),
\end{equation}
or equivalently, 
\begin{equation*}
  \label{eq:rho}
  l^2 rd-l(gr-r-d)+\delta  \geq 0.
\end{equation*}
This quadratic polynomial attains its minimum for 
$l_0=\frac{gr-r-d}{2rd}$.  The inequality
\eqref{eq:rho0} therefore holds for the closest integer to $l_0$, which is $\alpha$. This proves \eqref{eq:bound}. 
\end{proof}

\begin{remark}
For $\delta=0$, condition \eqref{eq:bound} reduces to \eqref{eq:rhocond} again.
\end{remark}

\begin{remark} \label{rem:kgonalnonnodal}
In the case $r=1$, setting $d=k$, condition \eqref{eq:bound} 
reads like \eqref{eq:boundintro}. Thus Theorem \ref{thm:noexist} proves the ``only if'' part of Theorem \ref{thm:main2}(i). However,
Theorem \ref{thm:noexist} does not require the curves to be nodal, whence 
\eqref{eq:boundintro} is also necessary for the nonemptiness of $V$ as in Theorem \ref{thm:boundingdim}.
\end{remark}

\begin{remark} \label{rem:optimusprime}
We spend a few words on the optimality of Theorem \ref{thm:necex}. 
The result is optimal for $r=1$ because the condition $d\geq r(r+1)$ is automatically fulfilled. It turns out that the theorem is optimal even for $r < p-1-d$. Indeed, in this range the inequality $\rho(p,r,d)\geq -r(r+2)$ is also {\it sufficient} for nonemptiness of $|L|^r_d$, as it implies the condition 
$d \geq r(r+1)$   in Theorem \ref{thm:main4}(ii). 

However, the result is not optimal when
$0 \leq p-1-d \leq r$, as shown by Examples \ref{cex1}--\ref{cex2} below.
In the Appendix we will see that the bound \eqref{eq:rhocond} in Theorem \ref{thm:necex} can be considerably improved for $0 \leq p-1-d \leq r$ (cf. Theorem \ref{thm:necexforte}) as in this range the bundles $\E_{C,\A}$ are forced to have nonvanishing $H^1$ by \eqref{eq:p0} and \eqref{eq:p3}. This behaviour never occurs on $K3$ surfaces, where Lazarsfeld-Mukai bundles have vanishing $H^1$. This stronger version of Theorem \ref{thm:necex} is postponed until the Appendix, since it is a bit technical and neither the result nor its proof are used in the rest of the paper. 
\end{remark}

\begin{example}\label{cex1}
For $r=2$, $d=4$ and $p=6$, inequality \eqref{eq:rhocond} is satisfied. However, any smooth curve possessing a $g^2_4$ is hyperelliptic and thus $|L|^2_4=|L|^1_2=\emptyset$ because $\rho(6,1,2)<-3$.
\end{example}

\begin{example}\label{cex3}
For $r=3$, $d=7$ and $p=10$, the inequality \eqref{eq:rhocond} is satisfied. Assume the existence of a curve $C\in |L|^3_7$. Since $|L|^3_6=|L|^1_2=\emptyset$, any $g^3_7$ on $C$ would define a birational map $C\to X\subset \mathbb{P}^3$, where $X$ is a nondegenerate space curve of degree $7$, contradicting  Castelnuovo's bound.
\end{example}

\begin{example}\label{cex2}
This example is of particular interest, since it hightlights yet another  major difference with the $K3$ setting. Let $r=2$, $d=5$ and $p=7$. Then the bound \eqref{eq:rhocond} is satisfied and any $g^2_5$ on a curve $C\in |L|^2_5$ should be base point free since $|L|^2_4=|L|^1_2=\emptyset$. We conclude that $|L|^2_5=\emptyset$ because the arithmetic genus of any plane quintic curve is $6$.

On the other hand, Theorem \ref{thm:main2}(i) will show that the locus $|L|^1_{1,2}$ is nonempty, that is, there is a nodal curve $C$ in $|L|$ whose normalization has genus $6$ and possesses a $g^1_2$, whence also a $g^2_4$. Pushing the latter down to $C$, we see that $C$ possesses a torsion free sheaf $\A$ with $d=\deg \A=5$ and $r=h^0(\A)-1=2$.

This phenomenon does not occur on  $K3$ surfaces: if $C$ is an irreducible curve on a $K3$ surface $S$ and $\A$ is a torsion free sheaf on $C$ of degree $d$, then $(C,\A)$ can be deformed to a pair $(C_1,A_1)$ such that $C_1\subset S$ is a smooth curve and $A_1\in G^r_d(C_1)$ with $r=h^0(\A)-1$ (cf. \cite[Prop. 2.5]{Go}).
\end{example}

\section{$k$-gonal loci on abelian surfaces} \label{sec:sevdeg}

The aim of this section is to study the $k$-gonal loci $|L|^1_{\delta,k}$ and prove the remaining part of Theorem \ref{thm:main2}. We again use the degenerate abelian surface $(S_0,L_0)$ and the same notation as in \S \ref{severi}.

Let $X=\nu(C)\in V(\alpha_0,\ldots,\alpha_{n-1})$ and let $\overline{X}$ be the stable model of the partial normalization of $X$ at its $\delta$ marked points. We identify the component $\Gamma\subset C$ that is not a fiber with $E$ and define the scheme: 
$$
G^1_k(\overline{X}):=\left\{\mathfrak{g}\in G^1_k(E)\;\Big\vert\; \dim\left|\mathfrak{g}\left(-P-P\+e^{\+l}\right)\right|>0 \textrm{ for all distinguished pairs }\left(P,P\+e^{\+l}\right)\right\}.
$$
\begin{definition}\label{defn:limkgonal}  
We define $V^k(\alpha_0,\ldots,\alpha_{n-1})$ to be 
the closed subset of  $V(\alpha_0,\ldots,\alpha_{n-1})$
consisting of curves $X=\nu(C)$ such that $G^1_k(\overline{X})\neq\emptyset$.
\end{definition}
As a consequence of Lemma \ref{lemma:stablemodel} and the theory of admissible covers (cf. Remark \ref{pencil}), we get:

\begin{lemma} \label{lemma:kgonale}
 Let $X=\nu(C)$ be a curve in $V( \alpha_0,\ldots,\alpha_{n-1})$.  Then $X$ lies in $|L_0|^1_{\delta,k}$ if and only if it lies in $V^k(\alpha_0,\ldots,\alpha_{n-1})$. In particular, $V^k( \alpha_0,\ldots,\alpha_{n-1})$  fills up one or more components of
$|L_0|^1_{\delta,k}$.
\end{lemma}

We now translate nonemptiness of $V^k( \alpha_0,\ldots,\alpha_{n-1})$ into a problem of intersection theory on $\Sym^2(E)$. First of all, we recall that the variety $\Pic^2(E)$ is isomorphic to $E$: indeed, fixing any point $e \in E$, we obtain distinct line bundles $\O_E(x+e)$ for all $x \in E$. Moreover, $q:\Sym^2(E)\to E\simeq \Pic^2(E)$ is a ruled surface with fiber over a point $x\in E$ given by the linear system $|x+e|$, which is of type $g^1_2$. We denote the class of the fibers by $\mathfrak{f}$, and the fiber over a point $x \in E$ by $\mathfrak{f}_x$. For each $y \in E$ one defines a section of $q$ by setting:
\[ \mathfrak{s}_y: = \{x+y \; | \: x \in E\}. \]
Let $\mathfrak{s}$ be its algebraic equivalence class. Since for $y \neq y'$ the sections $\mathfrak{s}_y$ and $\mathfrak{s}_{y'}$
intersect transversally in the point $y+y'$, then $\mathfrak{s}^2=1$. In particular, one has $K_{\Sym^2(E)} \equiv -2 \mathfrak{s}+\mathfrak{f}$.

 Now, we recall the definition of some curves in $\Sym^2(E)$, already introduced in \S \ref{severi}. Having fixed a positive integer $l$ and a point $y \in E$, we set: 
 \[
\mathfrak{c}_{y,l}:=\left\{x+(x\+y^{\+l})\in \Sym^2(E)\; |\, x\in E\right\}.
\]
For fixed $l$, these curves are all algebraically equivalent as $y$ varies. Clearly $\mathfrak{c}_{e_0,l} = \Delta\cong E$, the diagonal in $\Sym^2(E)$; the Riemann-Hurwitz formula thus yields $ \mathfrak{c}_{e_0,l} \cdot \mathfrak{f}=4$. Also note that $\mathfrak{c}_{y,l}\cdot \mathfrak{s}=2$, the two intersection points of $\mathfrak{c}_{y,l}$ with the section $\mathfrak{s}_y$ equal to $y+ y ^{\+l}$ and $ y ^{\ominus l}+y$. This gives
\[ \mathfrak{c}_{y,l} \equiv \Delta \equiv 4 \mathfrak{s} - 2\mathfrak{f}. \]

The Brill-Noether variety $G^1_k(E)$ is irreducible of dimension $2k-3$. To each $g^1_k$, call it $\mathfrak{g}$, on $E$ we associate a curve
\[ C_{\mathfrak{g}}:=\{ x+y \in \Sym^2(E) \; | \; \mathfrak{g}(-x-y) \geq 0 \} \]
in $\Sym^2(E)$. To compute the class of this curve, note that $C_{\mathfrak{g}} \cdot \Delta=2k$ again by Riemann-Hurwitz, and that $C_{\mathfrak{g}} \cdot \mathfrak{s} =k-1$ since there is a unique element of $\mathfrak{g}$ containing a fixed point $x\in E$. Hence:
\[ C_{\mathfrak{g}} \equiv (k-2)\mathfrak{s}+\mathfrak{f} =: C_k. \]
(Equivalently, the class of $C_{\mathfrak{g}}$ can be found by specializing  ${\mathfrak{g}}$ to a $g^1_2$ plus $(k-2)$ fixed points.) 

We have that $C_k \cdot \mathfrak{c}_{y,l}=2k$.

\begin{lemma} \label{lemma:trans}
  For general $y \in E$, the following are satisfied:
  \begin{itemize}
  \item[(i)] no curve $\mathfrak{c}_{y,l}$ is contained in a curve 
$C_{\mathfrak{g}}$;
 \item[(ii)] for general $\mathfrak{g}$, the curve $C_{\mathfrak{g}}$ intersects each $\mathfrak{c}_{y,l}$ transversally in $2k$ distinct points;
 \item[(iii)] in addition, none of these $2k$ points is fixed varying $\mathfrak{g}$.
  \end{itemize}
\end{lemma}

\begin{proof}
  By moving $y$, each $\mathfrak{c}_{y,l}$ moves in  a one-dimensional family containing the diagonal $\Delta$ obtained for $y=e_0$. Since $\Delta$ is not contained in any $C_{\mathfrak{g}}$, (i) follows. Similarly, it suffices to prove (ii)
for the intersection with $\Delta$ of a $C_{\mathfrak{g}}$, which are the ramification points of ${\mathfrak{g}}$. Hence (ii) follows, and (iii) is proved in the same way.
 \end{proof}

We can now prove the following:

\begin{prop} \label{prop:riempie}
If \eqref{eq:numcond}, \eqref{eq:nodes} and
\begin{equation}
\label{eq:numcond2} \alpha_{j} \leq 2k \; \; \mbox{for all} \; \;0\leq j\leq n-1 
\end{equation}
hold, then the variety 
$V^k( \alpha_0,\ldots,\alpha_{n-1})$
is nonempty and all irreducible components of $|L_0|^1_{\delta,k}$ contained in $V^k( \alpha_0,\ldots,\alpha_{n-1})$ have dimension equal to $\min\{2k-4,g-2\}$.  Furthermore, for a general curve $X=\nu(C)\in V^k( \alpha_0,\ldots,\alpha_{n-1})$, the scheme $G^1_k(\overline{X})$ is reduced of dimension $\max\{0,\rho(g,1,k)\}$. 
\end{prop}

\begin{proof}
We consider the following incidence variety, which is nonempty by \eqref{eq:numcond2} and Lemma \ref{lemma:trans}(ii): 
$$
I:=\left\{\left(\mathfrak{g},\{x_{l,j}\}_{\stackrel{0 \leq l \leq n-1}{1 \leq j \leq \alpha_l}}\right)\;\left\vert\right.\mathfrak{g} \in G^1_k(E),\; x_{l,j}=P_{l,j}+(P_{l,j}\+e^{\+ l})\, \mbox{distinct points of}\,\, C_\mathfrak g\cap \mathfrak{c}_{e,l+1}\right\}.
$$
Since a general fiber of the projection $p:I\to G^1_k(E)$ is finite and reduced by Lemma \ref{lemma:trans}(ii), the scheme $I$ itself is reduced. Let $q:I\to \Sym^{\alpha_0}(E) \x \cdots \x \Sym^{\alpha_{n-1}}(E)$ and $\pi:I\to \Sym^n(E)$ map a point $(\mathfrak{g},\{x_{l,j}\}_{l,j}) \in I$ to $\{P_{l,j}\}$ and $\O_E(\sum_{l,j} D_l(P_{l,j}))$, respectively. By proceeding as in the proof of Lemma \ref{lemma:v0}, one verifies that $\pi^{-1}|\O_E(\nu^*L_0)|$ has codimension $1$ in $I$. Indeed, given $\{P_{l,j}\} \in \Sym^{\alpha_0}(E) \x \cdots \x \Sym^{\alpha_{n-1}}(E)$ such that all 
divisors $x_{l,j}=P_{l,j}+(P_{l,j} \+e^{\+(l+1)})\in \mathfrak{c}_{e,l+1}$ belong to the same curve $C_{\mathfrak{g}}$, with $\mathfrak{g}\in G^1_k(E)$, then for any point $P \in E$ also the divisors \linebreak
$(P_{l,j}\+P) +(P_{l,j} \+e^{\+(l+1)} \+ P)$ all belong to the same curve $C_{\mathfrak{g'}}$ with $\mathfrak{g'}\in G^1_k(E)$, that is, $(\mathfrak{g}',\{x_{l,j}\+ P\}_{l,j})\in I$. More precisely, this shows that $\pi^{-1}|\O_E(\nu^*L_0)|\simeq \pi^{-1}|A|$ for all line bundles $A\in \Pic^{n}(E)$ (cf. Remark \ref{cata}) and in particular $\pi^{-1}|\O_E(\nu^*L_0)|$ is reduced. Furthermore, by Lemma \ref{lemma:unico}, the scheme $q(\pi^{-1}|\O_E(\nu^*L_0)|)$ is contained in $f(V( \alpha_0,\ldots,\alpha_{n-1}))$, where $f$ is the injection in the proof of Lemma \ref{lemma:v0}. Then Definition  \ref{defn:limkgonal} and Lemma \ref{lemma:kgonale} imply that $q(\pi^{-1}|\O_E(\nu^*L_0)|)$ equals $f(V^k( \alpha_0,\ldots,\alpha_{n-1}))$ and $\pi^{-1}|\O_E(\nu^*L_0)|$ is isomorphic to the scheme $\mathcal{G}^1_k(V^k(\alpha_0,\ldots,\alpha_{n-1}))$ parametrizing pairs $(X,\mathfrak g)$ with $X=\nu(C)\in V^k(\alpha_0,\ldots,\alpha_{n-1})$ and $\mathfrak{g}\in G^1_k(\overline{X})$.

When $g >2(k-1)$, this shows that $V^k( \alpha_0,\ldots,\alpha_{n-1})$ is nonempty and equidimensional of dimension $2k-4$; moreover, for a general $X\in V^k( \alpha_0,\ldots,\alpha_{n-1})$ the variety $G^1_k(\overline{X})$ is reduced and $0$-dimensional. When $g \leq 2(k-1)$, one has $V^k( \alpha_0,\ldots,\alpha_{n-1})=
V( \alpha_0,\ldots,\alpha_{n-1})$ and $G^1_k(\overline{X})$ is reduced of dimension $2k-4-(g-2)=\rho(g,1,k)$ for a general $X\in V( \alpha_0,\ldots,\alpha_{n-1})$. The fact that $V^k( \alpha_0,\ldots,\alpha_{n-1})$ is the union of irreducible components
of
${V^k_{|L_0|,\delta}(S_0)}$ follows from Lemma \ref{lemma:kgonale}.
\end{proof}

Recall Corollary \ref{cor:esisterel} and the notation therein. We have an $f$-relative Severi variety \linebreak $\phi_{|\L|,\delta}: \V_{|\L|,\delta} \to \DD$,
where $\V_{|\L|,\delta}:=\V_{\{\L\},\delta} \cap \PP(f_*\L)$, with fibers 
$V_{|L_t|,\delta}(S_t)$. 
One can  define the {\it relative $k$-gonal loci} 
$\{\L\}^1_{\delta,k} \sub \V_{\{\L\},\delta}$ and 
$|\L|^1_{\delta,k} \sub \V_{|\L|,\delta}$ in the obvious way, by extending the moduli map \eqref{eq:uno} to the whole of $\V_{\{\L\},\delta}$. Clearly, one has $\{\L\}^1_{\delta,k}=\V_{\{\L\},\delta}$ and $|\L|^1_{\delta,k}=\V_{|\L|,\delta}$ if $g \leq 2(k-1)$. Moreover, the same argument as in the proof of Proposition \ref{prop:expdimvk} shows that the expected dimensions of $\{\L\}^1_{\delta,k}$ and 
$|\L|^1_{\delta,k}$ are $\min\{2k-1,g+1\}$ and $\min\{2k-3,g-1\}$, respectively. We therefore have:

\begin{prop}\label {degargrel}
Let $V_0$ be a  component of $|L_0|^1_{\delta,k}$ of the expected dimension 
 $\min\{2k-4,g-2\}$. Then $V_0$ is contained in an irreducible component $\V$ of $|\L|^1_{\delta,k}$ dominating $\mathbb D$ and of dimension $\dim V_0+1$.
\end{prop}

\begin{cor} \label{rem:ram2} Let $(\alpha_0,\ldots,\alpha_{n-1})$ satisfy \eqref{eq:numcond}, \eqref{eq:nodes} and  
\eqref{eq:numcond2} and let $t\in\mathbb{D}^*$ be general. Then:
\begin{itemize}
\item [(i)]  $|L_t|^1_{\delta,k}$ has a component $V_t$ of the expected dimension $\min\{2(k-2),g-2\}$, such that the limit of $V_t$ when $S_t$ tends to $S_0$ is a component of $V^k( \alpha_0,\ldots,\alpha_{n-1})$;
\item [(ii)] for a general curve $C_t$ in $V_t$ with normalization $\widetilde{C}_t$, the variety $G^1_k(\widetilde{C}_t)$ is reduced of dimension $\max\{0,\rho(g,1,k)\}$;
\item [(iii)]  a general  $g^1_k$ in any component of $G^1_k(\widetilde{C}_t)$ has simple ramification;
\item [(iv)] the $\delta$ nodes of $C_t$ are non-neutral with respect to a general  $g^1_k$ in any component of $G^1_k(\widetilde{C}_t)$.
\end{itemize}
\end{cor} 

\begin{proof} By Proposition \ref {prop:riempie}, any component of $V^k( \alpha_0,\ldots,\alpha_{n-1})$ has dimension $\min\{2(k-2),g-2\}$. Then Proposition \ref{degargrel} yields (i). The last statement of Proposition \ref{prop:riempie} implies (ii). 

The stable model $\overline X$ of the partial normalization of a general curve $X$ in $V^k( \alpha_0,\ldots,\alpha_{n-1})$ at its $\delta$ marked points is obtained from $\Gamma\simeq E$ identifying $g-1$ pairs of points in a general $\mathfrak g=g^ 1_k$ on $E$ (see 
Lemmas \ref{lemma:stablemodel} and \ref{lemma:kgonale}).  The limit ramification points of a general element in $G^1_k(\overline{X})$ are 
 the ones of  $\mathfrak g$, which are simple as $\mathfrak g$ is general, in addition to the ones tending to the nodes of $\overline X$, which are simple by the definition of admissible cover. This proves (iii). 

Point (iv) can be similarly proven by degeneration to a general curve in $V^k(\alpha_1,\ldots,\alpha_n)$ and considering its admissible covers.  \end{proof}

We can now finish the proof of  Theorem \ref{thm:main2}. 

\renewcommand{\proofname}{Proof of Theorem \ref{thm:main2}}
\begin{proof}
We can assume that $\delta \leq p-2$ and $p\geq 3$. The only if part of (i) is Theorem \ref{thm:noexist}.
The rest   follows from  Corollaries \ref{cor:tuttecompdim} and \ref{rem:ram2} and the following lemma.
\end{proof}

\begin{lemma} \label{lemma:trovainteri}
  Let $n,k$ be integers with $n \geq 2$, $k \geq 2$ and let 
$0\leq \delta \leq p-2$ satisfy  \eqref{eq:boundintro}. 
Then there are integers $\alpha_{j}$ such that 
\eqref{eq:numcond}, \eqref{eq:nodes} and \eqref{eq:numcond2} hold.
\end{lemma}
\renewcommand{\proofname}{Proof}

\begin{proof}
The proof proceeds by induction on $\delta_0(n,k)\leq \delta \leq p-2$, with  $\delta_0(n,k)$ the minimal integer $\delta$ satisfying \eqref{eq:boundintro}. It  can be written explicitly in the following way: we first define 
\begin{equation*}
  \label{eq:defm}
 m= m(n,k):=\max\{l \in \ZZ \; | \; kl(l+1) \leq n\}, 
\end{equation*}
and
\begin{equation*}
  \label{eq:deft}
t=t(n,k):= \max\{l \in \ZZ \; | \; km(m+1)+l(m+1) \leq n\},  \;\text {i.e.},\; t= \Big\lfloor \frac{n}{m+1}\Big\rfloor-mk;
\end{equation*}
then, there is an integer $\lambda$ such that
 \begin{equation*} \label{eq:propp}
n= km(m+1) + t(m+1) +\lambda,
\end{equation*} 
and we notice that
$0 \leq t < 2k$ and $0  \leq  \lambda  \leq m$. 
An easy computation gives: 
\begin{equation*}
    \delta_0(n,k):=km(m-1)+tm+\lambda= \Big\lceil \frac{mn}{m+1}\Big\rceil-mk.
    \end{equation*}

We first provide integers $\alpha_j$ satisfying \eqref{eq:numcond2}, \eqref{eq:numcond} and \eqref{eq:nodes} for
$\delta_0(n,k)$.
 If $\lambda=0$, let  
\[ \alpha_{j}=  
\begin{cases}  
2k & \;\mbox{for} \; j=0,\ldots,m-1; \\
t & \;\mbox{for} \; j=m; \\
0 & \;\mbox{for} \; j >m.
\end{cases} 
\]
If $t =0$ and $\lambda >0$, let 
\[ \alpha_{j}=  
\begin{cases}  
2k & \;\mbox{for} \; j=0,\ldots,m-2;\\
2k-1 & \;\mbox{for} \; j=m-1;\\
1 & \;\mbox{for} \; j=m-1+\lambda;\\
0 & \;\mbox{otherwise}.
\end{cases} 
\]
Finally, if $t>0$ and $\lambda >0$, let 
\[ \alpha_{j}=  
\begin{cases}  
2k & \;\mbox{for} \; j=0,\ldots,m-1;\\
t-1 & \;\mbox{for} \; j=m;\\
1 & \;\mbox{for} \; j=m+\lambda;\\
0 & \;\mbox{otherwise}.
\end{cases} 
\]

The induction step consists in showing that, as soon some integers $\alpha_j$ satisfy \eqref{eq:numcond2}, \eqref{eq:numcond} and \eqref{eq:nodes} for $\delta=\delta_1<p-2$, there are integers $\alpha_j'$ satisfying \eqref{eq:numcond2}, \eqref{eq:numcond} and \eqref{eq:nodes} for $\delta=\delta_1+1$. Let $h$ be the largest integer such that $\alpha_h\geq 1$. Note that $h\neq n-1$ because otherwise we would have $\alpha_{n-1}=1$ and $\alpha_j=0$ for $j\neq n-1$ by \eqref{eq:numcond}, thus obtaining the contradiction $\delta_1=p-2$ by  \eqref{eq:nodes}. We set
\[ \alpha'_{j}=  
\begin{cases}  
\alpha_j-1 &  \; \mbox{if} \; j=h, \\ 
1  & \;\mbox{if} \; j=h+1, \\
\alpha_{j} & \;\mbox{otherwise},
\end{cases} 
\]
and one can easily check that these integers satisfy the desired properties.
\end{proof}

\renewcommand{\proofname}{Proof}

\begin{remark} \label{rem:cazzo}
  The fact that there is {\it one} component of $|L|^1_{\delta,k}$ satisfying (ii) in Theorem \ref{thm:main2} also follows from Corollary \ref{rem:ram2} and Lemma  \ref{lemma:trovainteri}. Corollary \ref{cor:tuttecompdim} is only needed to prove that (ii) holds on {\it all} components of $|L|^1_{\delta,k}$.
\end{remark}

\section{Linear series of type $g^r_d$ with $r\geq 2$ on smooth curves}\label{cuspidi}
In this section, we investigate the varieties $|L|^r_d$ and $\{L\}^r_d$ for $r\geq 2$, as defined in \S \ref{sec:lser}. For convenience, the Severi varieties $V_{\{L\},0}(S)$ and $V_{|L|,0}(S)$ will be denoted by $\{L\}_s$ and $|L|_s$, respectively; the same notation will be used for the degenerate abelian surface $(S_0,L_0)$.  Similarly, in the $f$-relative setting (cf. Corollary \ref{cor:esisterel}) we set $\{\L\}_s:=\V_{\{\L\},0}$ and  
$|\L|_s:=\V_{|\L|,0}=\V_{\{\L\},0} \cap \PP(f_*\L)$. 
When $S$ is smooth, we consider the schemes and morphisms
$\G^r_d(|L|) \to |L|_s$ and $\G^r_d(\{L\}) \to \{L\}_s$ with fiber over a smooth curve $C$ equal to $G^r_d(C)$.  

We proceed again by degeneration to $S_0$ and consider irreducible curves $X\in V(n,0,\ldots,0)=|L_0|_s$; such an $X$ is stable as it is obtained from an irreducible curve $\Gamma\simeq E$ lying on $R$, by identifying $n=p-1$ pairs of points $(P_1,P_1\oplus e)$, \ldots,$(P_{p-1},P_{p-1}\oplus e)$. We consider the Brill-Noether variety 
\begin{eqnarray*}
G^r_d(X):=\{\; \gg\in G^r_d(E) & | & \dim|\gg(-P_i-P_i\oplus e)|=\dim|\gg(-P_i)|= \\ & & \dim|\gg(-P_i\oplus e)|=r-1, \; \; \mbox{for} \; \; \,i=1,\ldots,p-1\}.
\end{eqnarray*}
Since $r\geq 2$, then $G^r_d(X)$ is non-compact and is open but not necessarily dense in the generalized Brill-Noether locus $\overline{G^r_d}(X)$, which also contains linear series $\gg\in G^r_{d-m}(E)$ for some $m\geq 1$ such that the condition $\dim|\gg(-P_i-P_i\oplus e)|=\dim|\gg(-P_i)|=\dim|\gg(-P_i\oplus e)|=r-1$ is satisfied only for $p-1-m$ of the indices $i$'s; in other words, $\overline{G^r_d}(X)$ parametrizes pairs $(B,V)$ where $B$ is a torsion free sheaf of rank $1$ and degree $d$ on $X$ and $V\in G(r+1, H^0(B))$.  By moving $X\in \{L_0\}_s$, we obtain the scheme 
\begin{eqnarray*}
\G^r_d(\{L_{0}(e,p)\}) :=  \{\gg\in G^r_d(E) &|& \exists\, P_1,\ldots,P_{p-1}\in E \; \mbox{such that} \; 
 \dim|\gg(-P_i-P_i\oplus e)| = \\ & & \dim|\gg(-P_i)|=\dim|\gg(-P_i\+e)|=r-1 \; \; \mbox{for} \; \; \,i=1,\ldots,p-1\},
\end{eqnarray*}
where we use the notation $L_0(e,p)$ in order to remember both the glueing parameter $e$ and the number of pairs of points we identify. Notice that the definition of $\G^r_d(\{L_0(e,p)\})$ makes sense even when $e$ is a torsion point. The variety $\G^r_d(|L_0(e,p)|)$ is defined in the same way with the further requirement that $P_1\+\cdots\+P_{p-1}$ is constant (indeed, when $e$ is not a torsion point, we have to impose that $P_1+\cdots+P_{p-1}\in |\O_E(\nu^*L_0(e,p))|$). Analogously, one defines the generalized Brill-Noether varieties $\overline{\G}^r_d(\{L_0(e,p)\})$  and $\overline{\G}^r_d(|L_0(e,p)|)$. Note that, being open in $\overline{\G}^r_d(|L_0(e,p)|)$, the scheme $\G^r_d(|L_0(e,p)|)$ has expected dimension equal to $p-2+\rho(p,r,d)$. We also denote by $\overline{|L_0(e,p)|}^r_d$ the image in $|L_0(e,p)|_s$ of the natural projection 
$$
\xymatrix{\pi_0:\overline{\G}^r_d(|L_0(e,p)|)\ar[r]& |L_0(e,p)|_s\subset |L_0(e,p)|,}$$ 
and by $|L_0(e,p)|^r_d$ the image via $\pi_0$ of $\G^r_d(|L_0(e,p)|)$. Note that $\overline{|L_0(e,p)|}^r_d$ is closed in $|L_0(e,p)|_s$ but not in $|L_0(e,p)|$.

If $e$ is not a torsion point, we look at the family of $(1,n)$-polarized abelian surfaces $f: \mathcal{S} \to \mathbb{D}$ with central fiber $S_0$ constructed in \S \ref{degenerazioni}.

\begin{prop}\label{IUIA}
If $\G^r_d(|L_0(e,p)|)$ has an irreducible component of the expected dimension $p-2+\rho(p,r,d)$ whose general point defines a birational map to $\mathbb{P}^r$, then the same holds for $\G^r_d(|L_t|)$ with $t\in\mathbb{D}$ general. 

Analogously, if $|L_0(e,p)|^r_d$ has an irreducible component $Z\subset V(n,0,\ldots,0)=|L_0(e,p)|_s$ of the expected dimension $\min\{p-2,p-2+\rho(p,r,d)\}$, then the same holds for $|L_t|^r_d$ with $t\in\mathbb{D}$ general.
\end{prop}

\begin{proof}
We use the existence of a scheme $s:\overline{\G}^r_d(|\L|)\to \mathbb{D}$ with general fiber equal to $\G^r_d(\vert L_t\vert)$ and central fiber $\overline{\G}^r_d(|L_0(e,p)|)$, along with a projection $\pi:\overline{\G}^r_d(|\L|)\to |\L|_s=\V_{|\L|,0}$. The image of $\pi$ defines a scheme $t:|\L|^r_d\to \mathbb{D}$ with general fiber given by $|L_t|^r_d$ and central fiber equal to $\overline{|L_0(e,p)|}^r_d$. As in the proof of Proposition \ref{prop:expdimvk}, one shows that the expected dimension of $|\L|^r_d$ is $\min\{p-1, p-1+\rho(p,r,d)\}$, and similarly  the  expected dimension of $\overline{\G}^r_d(|\L|)$ is
$p-1+\rho(p,r,d)$. 

Since $\G^r_d(|L_0(e,p)|)$ (respectively, $|L_0(e,p)|^r_d$) is open (not necessarily dense) in $\overline{\G}^r_d(|L_0(e,p)|)$ (resp. $\overline{|L_0(e,p)|}^r_d$), the statement follows from upper semicontinuity (as both $t$ and $s$ are locally of finite type) along with the fact that birational linear series form an open subscheme of $\overline{\G}^r_d(|\L|)$.
\end{proof}

In order to study $|L_0(e,p)|^r_d$ and $\G^r_d(|L_0(e,p)|)$ for a general $e\in E$, we perform a further degeneration, namely, we let $e$ approach the neutral element $e_0\in E$. This method proves helpful:

\begin{lemma}\label{neve}
If $\overline{\G}^r_d(|L_0(e_0,p)|)$ has an irreducible component of the expected dimension whose general element is birational, then the same holds for $\overline{\G}^r_d(|L_0(e,p)|)$ with $e\in E$ general. The dimensional statement still holds if one replaces $\overline{\G}^r_d(|L_0(e_0,p)|)$ and $\overline{\G}^r_d(|L_0(e,p)|)$ with $|L_0(e_0,p)|^r_d$ and $|L_0(e,p)|^r_d$, respectively.
\end{lemma}
\begin{proof}
The statement follows from the existence of a scheme $\overline{\G}^r_d\subset G^r_d(E)\times E$ of expected dimension $p+\rho(p,r,d)$
such that the fiber of the second projection $\overline{\G}^r_d\to E$ over a point $e\in E$ coincides with $\overline{\G}^r_d(|L_0(e,p)|)$.
\end{proof} 
 
The study of the Brill-Noether locus $|L_0(e_0,p)|^r_d$ translates into a problem of linear series with prescribed ramification; we recall some basic theory, first developed by Eisenbud and Harris in \cite{EH1,EH2}. Let $C$ be a smooth irreducible curve of genus $g$ and fix $\gg=(A,V)\in G^r_d(C)$ and $P\in C$. The vanishing sequence $\underline{a}(\gg,P)=(a_0(\gg,P),\ldots,a_{r}(\gg,P))$ of $\gg$ at $P$ is obtained by ordering increasingly the set $\{\mathrm{ord}_P\sigma\}_{\sigma\in V}$ while the ramification sequence of $\gg$ at $P$ is defined as: $$\underline{b}(\gg,P):=\underline{a}(\gg,P)-(0,1,\ldots,i,\ldots,r).$$ Having fixed $m$ ramification sequences $\underline{b}^1,\ldots,\underline{b}^m$ and $m$ points $P_1,\ldots,P_{m}\in C$, one defines the variety of linear series with ramification at least $\underline{b}^i$ in $P_i$  as
$$G^r_d(C,( P_i,\underline{b}^i)_{i=1}^m):=\{\gg\in G^r_d(C)\;|\; \underline{b}(\gg,P_i)\geq \underline{b}^i\textrm{ for }i=1,\ldots,m\};$$
this has expected dimension equal to the {\it adjusted Brill-Noether number} $\rho(g,r,d)-\sum_{i=1}^m\sum_{j=0}^rb^i_j$. 

The condition $\gg\in\G^r_d(\{L_0(e_0,p)\})$ is equivalent to the existence of $p-1$ points $P_1,\ldots,P_{p-1}$ such that $\underline{b}(\gg,P_i)=(0,1,\ldots,1)$; in particular, $\gg$ should lie in the variety $G^r_d(E,(P_i,(0,1,\ldots,1))_{i=1}^{p-1})$, which has expected dimension equal to $$\rho(1,r,d)-r(p-1)=\rho(p,r,d).$$ If $\gg$ is base point free and defines a birational map $\phi_{|\mathfrak{g}|}:E\to\mathbb{P}^r$, we are requiring that the curve $\phi_{|\gg|}(E)$ has $p-1$ cusps at the images of the points $P_i$.
\begin{remark}\label{warwick}
Every linear series $\gg$ of type $g^r_d$ on a smooth curve $C$ of genus $g$ satisfies
$$
\sum_{P\in C}\sum_{j=0}^rb_j(\gg,P)=(r+1)d+(r+1)r(g-1)
$$
by the Pl\"ucker Formula (cf. \cite{EH1}). Applying this to the elliptic curve $E$, we obtain precisely the condition \eqref{liscio} as  
necessary condition for nonemptiness of $G^r_d(E,(P_i,(0,1,\ldots,1))_{i=1}^{p-1})$. 
\end{remark}

We first deal with the cases where $\rho(p,r,d)\geq 0$.
\begin{thm}\label{minestrone}
Let $P_1,\ldots,P_{p-1}\in E$ be general. Then $G^r_d(E,(P_i,(0,1,\ldots,1))_{i=1}^{p-1})\neq \emptyset$ if and only if $\rho(p,r,d)\geq 0$. Moreover, as soon as it is nonempty, $G^r_d(E,(P_i,(0,1,\ldots,1))_{i=1}^{p-1})$ has the expected dimension $\rho(p,r,d)$, and a general $\gg\in G^r_d(E,(P_i,(0,1,\ldots,1))_{i=1}^{p-1})$ satisfies
$
\underline{b}(\gg,P_i)=(0,1,\ldots,1)$ for $i=1,\ldots,p-1$.
\end{thm}
\begin{proof}
The dimension statement is \cite[Thm. 1.1] {EH2}. Furthermore, nonemptiness is equivalent to $(\sigma_{0,1,\ldots,1})^p\neq 0$
in the cohomology ring $H^*(\mathbb{G}(r,d),\mathbb{Z})$ (cf. \cite[Thm. 5.42]{HM}). This condition is the same as the one ensuring existence on a general pointed curve $[C,P]\in\M_{p-1,1}$ of a $g^r_d$ with a cusp at $P$  (cf. \cite[Cor. 5.43]{HM}) and, by the Littlewood-Richardson rule, it is equivalent to:
$$
(p-1-d+r)+\sum_{j=1}^r(p-d+r)\leq p-1,
$$
that is, $\rho(p,r,d)\geq 0$. 

Let $\gg\in G^r_d(E,(P_i,(0,1,\ldots,1))_{i=1}^{p-1})$ be general. If we had $\underline{b}(\gg,P_k)>(0,1,\ldots,1)$ for some $k$, then $\gg$ would lie in the variety $G_1:=G^r_d(E,(P_i,(0,1,\ldots,1))_{i\neq k}, (P_k,\underline{b}(\gg,P_k)))$; this is a contradiction as $\dim\,G_1=\rho(p,r,d)+r-\sum_{j=0}^r b_j(\gg,P_k)<\dim G^r_d(E,(P_i,(0,1,\ldots,1))$ by \cite[Thm. 1.1] {EH2}.
\end{proof}

\begin{cor}\label{iuia}
If $\rho(p,r,d)\geq 0$, then $\G^r_d(|L_0(e_0,p)|)$ dominates $|L_0(e_0,p)|$; furthermore, every irreducible component of $\overline{\G}^r_d(|L_0(e_0,p)|)$ dominating $|L_0(e_0,p)|$ lies in the closure of $\G^r_d(|L_0(e_0,p)|)$ and has the expected dimension $p-2+\rho(p,r,d)$. In particular, if $X\in |L_0(e_0,p)|_s$ is general, then $G^r_d(X)$ is dense in $\overline{G^r_d}(X)$ and $\dim\overline{G^r_d}(X)=\rho(p,r,d)$.

If instead $\rho(p,r,d)<0$, then $\overline{|L_0(e_0,p)|}^r_d$ has codimension at least $1$ in $|L_0(e_0,p)|$.
\end{cor}
\begin{proof}
The scheme $\G^r_d(|L_0(e_0,p)|)$ is a fiber of the map $\G^r_d(\{L_0(e_0,p)\})\to \Pic^{p-1}(E)\simeq E$ sending a linear series $\gg$ to the line bundle $\O_E(P_1+\cdots+P_{p-1})$.  By varying the points $P_1,\ldots, P_{p-1}\in E$, Theorem \ref{minestrone} then ensures the existence of an irreducible component of $\G^r_d(|L_0(e_0,p)|)$ dominating $|L_0(e_0,p)|$ when $\rho(p,r,d)\geq 0$ and implies that any such component has the expected dimension. If instead $\rho(p,r,d)<0$, then $|L_0(e_0,p)|^r_d$ has codimension at least $1$ in $|L_0(e_0,p)|$. 

We claim that $G^r_d(X)$ is dense in $\overline{G^r_d}(X)$ if $X\in |L_0(e_0,p)|_s$ is general; this would conclude the proof. If $\gg\in \overline{G^r_d}(X)\setminus G^r_d(X)$, then $\gg\in G^r_{d-m}(\widetilde{X})$ where $\widetilde{X}$ is a partial normalization of $X$ at $m$ points, say $P_{p-m},\ldots,P_{p-1}\in E$. In other words, $\gg\in G^r_{d-m}(E,(P_i,(0,1,\ldots,1))_{i=1}^{p-1-m})$. As $P_1,\ldots,P_{p-1-m}\in E$ are general, Theorem \ref{minestrone} yields that $G^r_{d-m}(E,(P_i,(0,1,\ldots,1))$ is nonempty as soon as $\rho(p-m,r,d-m)\geq 0$ and, if nonempty, it has the expected dimension $\rho(p-m,r,d-m)$. The inequality $\rho(p-m,r,d-m)<\rho(p,r,d)$ yields our claim.
\end{proof}

The proof of Theorem \ref{thm:main4}(i) is now straightforward.

\renewcommand{\proofname}{Proof of Theorem \ref{thm:main4}(i)}

\begin{proof}
Combine Corollary \ref{iuia} with Lemma \ref{neve} and Proposition \ref{IUIA}. Since the projection $p:\overline{\G}^r_d(|\L|)\to \mathbb{D}$ is not proper (indeed, it factorizes through the map $\V_{|\L|,0}=|\L|_s\to\mathbb{D}$ and $|\L|_s$ is open in $\mathbb{P}(f_*\L)$), in order to prove the emptiness statement one should also verify that, if a component $\G$ of $\overline{\G}^r_d(|\L|)$ dominates $|L_t|$ for a general $t\in \mathbb{D}$, then it also dominates $|L_0|$ and hence contains a component of $\G^r_d(|L_0(e,p)|)$. This follows because a component $Z$ of $\overline{|\L|^r_d}$ (where $\overline{|\L|^r_d}$ denotes the closure of $|\L|^r_d$ in $\mathbb{P}(f_*\L)$) containing the whole $|L_t|$ for a general $t\in \mathbb{D}$ also contains the whole $|L_0(e,p)|$. 
\end{proof}

\renewcommand{\proofname}{Proof}

We will now focus on the cases where $\rho(p,r,d)<0$. 
Under the further assumption that $d\geq r(r+1)$, we will construct a component of $\G^r_d(|L_{0}(e_0,p)|)$ of the expected dimension. This is done by means of the duality for nondegenerate curves in $\mathbb{P}^r$, first discovered by Piene \cite{Pi}; we recall the main results and refer to \cite{Do,KS,Pi} for details. 

Let $Y\subset\mathbb{P}^r$ be a nondegenerate curve of degree $d$ and denote by by $f:\widetilde{Y}\to Y$ the normalization map. The dual map $f^\vee: \widetilde{Y}\to (\mathbb{P}^r)^\vee$ assigns to a point $P\in \widetilde{Y}$ the osculating hyperplane of $Y$ at $f(P)$; the image $Y^\vee:=f^\vee(\widetilde{Y})$ is then called the dual curve of $Y$. Let $\gg$ be the linear series on $\widetilde{Y}$ defining $f$ and, for any point $P\in \widetilde{Y}$, rewrite the ramification sequence of $\gg$ at $P$ as
$$
\underline{a}(\gg,P)=(0,k_1(P),\ldots,\sum_{j=1}^{i}k_j(P),\ldots,\sum_{j=1}^{r}k_j(P)).
$$
By setting $k_i=\sum_{P\in \widetilde{Y}} k_i(P)$, we obtain an $r$-tuple of nonnegative integers $(k_1,\ldots,k_r)$ associated with $\gg$. The $r$-tuple $(k_1^*,\ldots,k_r^*)$ is defined analogously starting from the linear series $\gg^*$ on $\widetilde{Y}$ corresponding to $f^\vee$. Having fixed our notation, we recall the duality theorem \cite[Thm. 5.1]{Pi}.
\begin{prop}\label{piene}
The following hold:
\begin{itemize}
\item[(i)] the dual curve $Y^\vee$ has degree $d^*:=rd-\sum_{i=1}^{r-1}(r-i)k_i$;
\item[(ii)] for any point $P\in \widetilde{Y}$ and $1\leq i\leq r$, one has $k_i^*(P)=k_{r+1-i}(P)$;
\item[(iii)] $(f^\vee)^\vee=f$.
\end{itemize}
\end{prop}
\begin{remark} Notice that our indices are rescaled in comparison with Piene's ones. Concerning item (ii), Theorem 5.1 in \cite{Pi} only states that $k_i^*=k_{r+1-i}$; the fact that such equality holds pointwise can be easily deduced either from Piene's proof or from \cite[pf. of Thm. 19]{KS}.
\end{remark}
Assume now that $\gg\in G^r_d(E)$ is base point free and defines a birational map $\phi: E\to X\subset\mathbb{P}^r$, where $X$ has degree $d$ and $p-1$ cusps. Also assume that, outside of the points $P_1,\ldots,P_{p-1}$ mapping to the cusps of $X$, the linear series $\gg$ has only simple ramification; then, by Remark \ref{warwick}, $\gg$ has $a$ points $Q_1,\ldots, Q_{a}$ of simple ramification, where $a:=(r+1)d-r(p-1)$. Notice that a general $\gg\in \G^r_d(|L_{0}(e_0,p)|)$ is expected to satisfy our assumptions. With the above notation, the cuspidal points satisfy $k_1(P_j)=1$ and $k_i(P_j)=0$ for $i\neq 1$, while the other ramification points satisfy $k_i(Q_j)=0$ for $i\neq r$ and $k_r(Q_j)=1$. Therefore, our hypotheses on $\gg$ yield $k_1=p-1$, $k_r=a$ and $k_i=0$ for $2\leq i\leq r-1$. By Proposition \ref{piene}, the dual curve $X^\vee$ of such an $X$ would have degree $d^*=rd-(r-1)(p-1)$ and $\gg^*$ would have $a$ ordinary cusps at the points $Q_1,\ldots,Q_a$ and simple ramification at $P_1,\ldots,P_{p-1}$. The existence of $\gg^*$ can be proven more easily than that of $\gg$: 
\begin{thm}\label{zurigo}
Let $p-1\geq r(r+2)$ and $\rho(p,r,d)\geq-r(r+2)$. Then $\G^r_d(|L_{0}(e_0,p)|)$ is nonempty with at least one irreducible component $\G$ of the expected dimension. A general $\gg\in \G$ is birational and has $p-1$ ordinary cusps and only ordinary ramification elsewhere. 
\end{thm}
\begin{proof}
By duality, it is enough to prove existence of a component $\G^*$ of  $\G^r_{d^*}(|L_{0}(e_0,a+1)|)$ such that a general $\gg^* \in\G^*$ is birational and has $a$ ordinary cusps and $p-1$ ordinary ramification points. Easy computations yield
$$
\rho(a+1, r, d^*)=\rho((r+1)d-r(p-1)+1,r,rd-(r-1)(p-1))=p-1-r(r+2).
$$
Hence, as soon as $p-1\geq r(r+2)$, Corollary \ref{iuia} implies the existence of  a component $\G^*$ of  $\G^r_{d^*}(|L_{0}(e_0,a+1)|)$ of the expected dimension 
$$
a-1+\rho(a+1, r, d^*)=(r+1)d-(r-1)(p-1)-1-r(r+2)=p-2+\rho(p,r,d).$$
In order to show that a general $\gg^*\in \G^*$ defines a birational map, has $a$ ordinary cusps and only ordinary ramification elsewhere, we proceed by degeneration to a rational curve with $a+1$ cusps as in \cite{EH}. By \cite[p. 384]{EH}, there exists a nondegenerate rational curve $X_0$ of degree $d^*$ in $\mathbb{P}^r$ having $a+1$ cusps because $\rho(a+1, r, d^*)\geq 0$. Since a cusp is locally smoothable, there is a flat family of curves $\mathcal{X}\to B$ over a $1$-dimensional base $B$ with central fiber $X_0$ and general fiber $X_b$ having $a$ cusps and normalization given by the elliptic curve $E$. We can also assume that for a general $b\in B$ the curve $X_b$ lies in $|L_{0}(e_0,a+1)|$ and the points of $E$ mapping to the cusps of $X_b$ are general. We consider the relative scheme $\G^r_{d^*}(\mathcal{X})\to B$ with fiber over a general $b\in B$ given by $G^r_{d^*}(X_b)$ and central fiber equal to $$G^r_{d^*}(X_0):=\{\gg\in G^r_{d^*}(\mathbb{P}^1)\;|\; \dim|\gg(-2y_i)|=\dim|\gg(-y_i)|=r-1,\,i=1,\ldots,a+1\},$$
where $y_1,\ldots, y_{a+1}$ are the points of $\mathbb{P}^1$ sent to the cusps of $X_0$ by the normalization map. We have already shown the existence of a component $\G^*(\mathcal{X})$ of $\G^r_{d^*}(\mathcal{X})$ of relative dimension $\rho(a+1,r,d^*)$. Since $G^r_{d^*}(X_0)$ is dense in $\overline{G}^r_{d^*}(X_0)$ by \cite[Thm. 4.5]{EH}, a general point $\gg_0$ of the fiber $\G^*(X_0)$ over $0$ lies in $G^r_{d^*}(X_0)$. By \cite[Thm. 3.1]{EH}, such a $\gg_0$ defines a birational map and has $a+1$ ordinary cusps at $y_1,\ldots, y_{a+1}$ and only ordinary ramification elsewhere. In order to prove the same properties for a general $\gg$ in a general fiber $\G^*(X_b)$, one proceeds exactly as in \cite[Props. 5.5, 5.6 and 5.7]{EH}.
\end{proof}
Now, we set $p':=\max\{p\;|\; \rho(p,r,d)\geq -r(r+2)\}$, that is,
$$
p'=1+\frac{(r+1)d-a'}{r},\textrm{  with } 0\leq a'\leq r-1.
$$
By definition, $\G^r_d(\{L_{0}(e_0,p')\})\subseteq \G^r_d(\{L_{0}(e_0,p\})$ for all $p$ such that $\rho(p,r,d)\geq -r(r+2)$. Furthermore, the codimension of $\G^r_d(\{L_{0}(e_0,p+1)\})$ in $\G^r_d(\{L_{0}(e_0,p)\})$ is at most $r-1$. Therefore, as soon as $\G^r_d(\{L_{0}(e_0,p')\})$ has an irreducible component of the expected dimension, then the same holds for $\G^r_d(\{L_{0}(e_0,p)\})$ for any $p<p'$. The inequality $p'-1\geq r(r+2)$ can be written as $$d\geq r(r+1)+(a'-r)/(r+1)$$ and this is equivalent to the requirement $d\geq r(r+1)$ because $0\leq a'\leq r-1$.

\begin{cor}\label{tardi}
Let $d\geq r(r+1)$ and $-r(r+2)\leq\rho(p,r,d)<0$. Then $\G^r_d(|L_{0}(e_0,p)|)$ is nonempty with at least one irreducible component $\G$ of the expected dimension $p-2+\rho(p,r,d)$ such that a general $\mathfrak{g}\in \G$ defines a birational map to $\mathbb{P}^r$. In particular, $|L_{0}(e_0,p)|^r_d$ has an irreducible component of the expected dimension, too.
\end{cor}

We can now conclude the proofs of Theorems \ref{thm:main4} and  \ref{thm:main5}.

\renewcommand{\proofname}{Proof of Theorem \ref{thm:main4}(ii),(iii)}

\begin{proof}
The condition $\rho(p,r,d)\geq -r(r+2)$ is necessary for nonemptiness of $|L|^r_d$ by \eqref{eq:rhocond}. Under the further assumption that $d\geq r(r+1)$, Corollary \ref{tardi}, Lemma \ref{neve} and Proposition \ref{IUIA} ensure the existence of a component $Z$ as in (ii). 
\end{proof}

\renewcommand{\proofname}{Proof of Theorem \ref{thm:main5}}

\begin{proof}
Let $S$ be a general abelian surface with polarization $L$ of type $(1,n)$, where $n=p-1$, and let $\psi:|L|_s\to \M_p$ be the moduli map. By Theorem \ref{thm:main4}(ii), $|L|^r_d$ has an irreducible component $Z$ of the expected dimension $p-2+\rho(p,r,d)$. Let $\M$ be an irreducible component of $\M^r_{p,d}$ such that $\psi(Z)$ is an irreducible component of $\M\cap \psi(|L|)$. The dimensional statement for $\M$ follows because all inequalities in (\ref{muffin}) must be equalities. We consider an irreducible component $\G_1$ of $\G^r_d(|L|)$ mapping finitely onto $Z$ such that, for a general $(C,\gg)\in\G_1$, the linear series $\gg$ defines a birational map to $\mathbb{P}^r$; the existence of such a $\G_1$ is ensured by Proposition \ref{IUIA}, Lemma \ref{neve} and Theorem \ref{zurigo}. One finds an irreducible component $\G$ of $\G^r_d$ such that $\G_1=\G\times_\M Z$ and easily verifies the desired properties for $\G$.
\end{proof}

\begin{remark} \label{rem:nonconst}
  Theorem \ref{thm:main4}(i)-(ii) implies that, in contrast to the $K3$ case, neither the gonality nor the Clifford index of smooth curves in $|L|$ is constant.

Indeed, by (i), a general curve in $|L|$ has the gonality and Clifford index of a general genus $p$ curve, namely, $k_{gen}:=\lfloor \frac{p+3}{2} \rfloor$ and
$c_{gen}:=\lfloor \frac{p-1}{2} \rfloor$, respectively. At the same time, by (ii),
there are smooth curves in $|L|$ carrying linear series of type $g^1_k$ with $k=\lfloor \frac{p+1}{2}\rfloor$, whence their gonality is at most $\lfloor \frac{p+1}{2}\rfloor$ and their Clifford index at most $\lfloor \frac{p+1}{2}\rfloor-2=\lfloor \frac{p-3}{2}\rfloor$.

 We can be more precise. Any smooth curve $C$ satisfies $\gon C-3 \leq \Cliff C \leq \gon C-2$, with equality on the right if the curve carries  only finitely many pencils of minimal degree \cite{CM}. Theorem \ref{thm:main2} for $\delta=0$ then implies that, if $p$ is even, then there is a codimension-$2$ family of smooth curves in $|L|$ of gonality $k_{gen}-1$ and Clifford index $c_{gen}-1$. If instead $p$ is odd, then there is a codimension-$1$ family of smooth curves in $|L|$ of gonality $k_{gen}-1$ and Clifford index $c_{gen}-1$, and a codimension-$3$ family
of gonality $k_{gen}-2$ and Clifford index $c_{gen}-2$. 
\end{remark}

\begin{remark}\label{menomale}
The inequality $d\geq r(r+1)$ implies that any $p$ with $-r(r+2)\leq\rho(p,r,d)<0$ satisfies: $$\expdim |L_{0}(0,p)|^r_d \geq \expdim |L_{0}(0,p')|^r_d=p'-2+\rho(p',r,d)\geq 0;$$ in fact, this condition is clearly necessary for the existence of a component of $|L_{0}(0,p)|^r_d$ having the expected dimension. 
\end{remark}
%\begin{remark}
%For fixed $r$ and $d$, set $p_{\min}:=\min\{p\,|\,\rho(p,r,d)<0\}$. One can easily show that the inequality $d\geq r(r+1)$ is equivalent to the requirement $d\leq p_{\min}-1$.
%\end{remark}

\begin{remark}\label{confronto}
Components of $\M^r_{p,d}$ of the expected dimension under certain conditions on $p,r,d$ such that $\rho(p,d,r)<0$  have been constructed by various authors following the foundational  works of Sernesi \cite {Se} and Eisenbud-Harris \cite{EH1} (see \cite{BE,P,L1,L2}), most recently by 
Pflueger \cite{Pf}. 

One can explicitly check that the  inequalities $d \geq r(r+1)$ and $\rho(p,r,d) \geq -r(r+2)$  appearing in our Theorem \ref{thm:main5} are weaker than the ones appearing in the above cited papers for infinitely many triples $(p,r,d)$. This is for instance the case when $r \geq 4$ and $r^2+2r-2 \leq p \leq  r^2+7r+7$.  Thus, in these cases the component of $\M^r_{p,d}$ detected by Theorem \ref{thm:main5} was heretofore unknown.

We also remark that our method of proof by degeneration to cuspidal elliptic curves and by means of Piene's duality is completely different, and in many ways simpler, than Eisenbud and Harris' limit linear series and Sernesi's original approach of attaching rational curves. 
\end{remark}

\renewcommand{\proofname}{Proof}

\begin{appendices}
\section{Appendix: A stronger version of Theorem \ref{thm:necex}}

We will prove a strenghtening of Theorem \ref{thm:necex} for $r \geq p-1-d$. In this range the bundles $\E_{C,\A}$ have  nonvanishing $H^1$ and the existence of stable extensions yields a stronger bound than \eqref{eq:rhocond}. Also note the slightly stronger assumption on $[C]$, which is no longer an open condition in the moduli space of polarized abelian surfaces, but holds off a countable union of proper closed subsets.

\begin{thm} \label{thm:necexforte}
Let $C$ be a reduced and irreducible curve of arithmetic genus $p=p_a(C)$ on an abelian surface $S$ such that $[C]$ generates $\NS(S)$. 
Assume that $C$ possesses a globally generated torsion free rank one sheaf $\A$ such that $\deg \A=d \leq p-1$, $h^0(\A)=r+1$ and $r \geq p-1-d$. Set $\gamma:=\lfloor \frac{r}{p-1-d} \rfloor$ if $d < p-1$.  Then one has:
  \begin{equation}
    \label{eq:rhocondforte}
   \rho(p,r,d) +r(r+2) \geq  \begin{cases} 
\frac{1}{2}r(r+1), & \mbox{if} \; \; d = p-1, \\ 
   \gamma\Big(p-1-d\Big)\Big(r+1-\frac{1}{2} (p-1-d)(\gamma+1)\Big)   & \mbox{if} \; \; r \geq p-1-d >0.
\end{cases}
  \end{equation}
\end{thm}
The proof needs the following technical result:

\begin{prop} \label{prop:stabile}
Let $\E$ be a vector bundle on an abelian surface $S$ satisfying:
\begin{itemize}
\item[(i)] $[c_1(\E)]$ generates $\NS(S)$; 
\item[(ii)] $\E$ is generically globally generated;
\item[(iii)] $H^2(\E)=0$.
\end{itemize}
Let $N \geq 0$ be an integer such that, for $i=1, \ldots, N,$ there exists a sequence of ``universal extensions''
\begin{equation}
  \label{eq:univext}
 \xymatrix{
0 \ar[r] & \O_S^{\+ h_i} \ar[r] & \E_i \ar[r] & \E_{i-1} \ar[r] & 0,} 
\end{equation}
where $\E_0:=\E$, $h_i:=h^1(\E_{i-1})$ and the coboundary map $H^1(\E_{i-1}) \to H^2(\O_S^{\+ h_i}) \cong \CC^{h_i}$ is an isomorphism (this condition is empty for $N=0$).

Then each $\E_i$ is stable with respect to any polarization and satisfies $H^2(\E_i)=0$.
\end{prop}

\begin{proof}
We proceed by induction on $i$. 

Let $i=0$. Then $H^2(\E_0)=H^2(\E)=0$ by assumption (iii). 
Assume that $\E_0=\E$ is not stable. Consider any destabilizing sequence 
 \begin{equation} \label{eq:dest}
    \xymatrix{
0 \ar[r] & \M \ar[r] & \E \ar[r] & \Q \ar[r] & 0,}
 \end{equation}
where $\M$ and $\Q$ are torsion free sheaves of positive rank; this gives $c_1(\E)=c_1(\M)+c_1(\Q)$. For any ample line bundle $H$ on $S$, we have that $c_1(\M).H >0$ because $\M$ destabilizes $\E$, 
and $c_1(\Q).H>0$ by Lemma \ref{lemma:fact}, because $H^2(\Q)=0$ and $\Q$ is globally generated off a codimension one set by \eqref{eq:dest} and assumptions (ii)-(iii).
This contradicts assumption (i).

Now assume that $i>0$ and $\E_{i-1}$ is $H$-stable with $H^2(\E_{i-1})=0$. 
The fact that $H^2(\E_i)=0$ is an immediate consequence of \eqref{eq:univext} and the coboundary map being an isomorphism. 
If $\E_i$ is not  $H$-stable, then we have a destabilizing sequence
 \begin{equation} \label{eq:desti}
    \xymatrix{
0 \ar[r] & \M \ar[r] & \E_i \ar[r] & \Q \ar[r] & 0,}
\end{equation}
with $\M$ and $\Q$ torsion free sheaves of positive rank. Let $\M'$ denote the image of the composition $\M \to \E_i \to \E_{i-1}$ of maps from \eqref{eq:desti} and \eqref{eq:univext}. Then we have a commutative diagram with exact rows and columns:
\begin{equation}
  \label{eq:diagrammone}
    \xymatrix{
& 0  \ar[d] & 0  \ar[d] & 0  \ar[d] &  \\ 
0 \ar[r] & \K \ar[r]  \ar[d] & \O_S^{\+ h_i}\ar[r]  \ar[d] & \K' \ar[r]  \ar[d] & 0 \\ 
0 \ar[r] & \M \ar[r]  \ar[d] & \E_i \ar[r]  \ar[d] & \Q \ar[r]  \ar[d] & 0 \\ 
0 \ar[r] & \M' \ar[r]  \ar[d] & \E_{i-1} \ar[r]  \ar[d] & \Q' \ar[r]  \ar[d] & 0, \\ 
& 0  & 0 & 0 & 
}
\end{equation}
defining $\K$, $\K'$ and $\Q'$. Since $\K'$ is globally generated, we have $c_1(\K).H=-c_1(\K').H \leq 0$. Hence $c_1(\M').H=c_1(\M).H-c_1(\K).H \geq c_1(\M).H >0$, as $\M$ destabilizes $\E_i$. In particular, as $\M'$ is torsion free, we have
$\rk \M' >0$. 

If $\rk \Q' >0$, then $\rk \M' < \rk \E_{i-1}$. As $\E_{i-1}$ is $H$-stable, we must have:
\[ \frac{ c_1(\M').H}{\rk \M'} < \frac{c_1(\E_{i-1}).H}{\rk \E_{i-1}}. \]
In particular, $0 < c_1(\M').H < c_1(\E_{i-1}).H$, so that $c_1(\Q').H >0$.
But then $c_1(\E)=c_1(\E_{i-1})= c_1(\M')+c_1(\Q')$ with both $c_1(\M').H >0$ and
$c_1(\Q').H >0$, contradicting hypothesis (i). 

Hence we have $\rk \Q'=0$ and $c_1(\Q')$, if nonzero, is represented by the effective cycle of the $1$-dimensional support of $\Q'$. Then $c_1(\Q).H=c_1(\K').H+c_1(\Q').H\geq 0$ and strict inequality follows from Lemma \ref{lemma:fact} because $h^2(\Q)=h^2(\E_i)=0$ and $\Q$ is globally generated off the $1$-dimensional support of $\Q'$, as $\K'$ is globally generated. 
As $\M$ destabilizes $\E_i$, we get that
$c_1(\E)=c_1(\E_{i})= c_1(\M)+c_1(\Q)$ with both $c_1(\M).H >0$ and
$c_1(\Q).H >0$, again contradicting hypothesis (i). 
\end{proof}

\renewcommand{\proofname}{Proof of Theorem \ref{thm:necexforte}}

\begin{proof}
Consider the vector bundle $\E_0:=\E_{C,\A}$ from  the proof of Theorem \ref{thm:necex}, which satisfies conditions (i)-(iii) of Proposition \ref{prop:stabile}. Moreover 
 \eqref{eq:p0} and \eqref{eq:p3} imply 
 \begin{equation}
   \label{eq:p6}
h_1:= h^1(\E_{C,\A})=  h^0(\E_{C,\A})-\chi(\E_{C,\A}) \geq r+2+d-p >0.
 \end{equation}   
Therefore, we have a ``universal extension''
\[
 \xymatrix{
0 \ar[r] & \O_S^{\+ h_1} \ar[r] & \E_1 \ar[r] & \E_{0} \ar[r] & 0,} 
\]
 where the coboundary map $H^1(\E_{0}) \to H^2(\O_S^{\+ h_1}) \cong \CC^{h_1}$ is an isomorphism.
 
If $h_2:=h^1(\E_1) >0$, we can iterate the construction. Hence, there is an integer $N >0$ and a sequence of universal extensions as in \eqref{eq:univext},
where the coboundary maps $H^1(\E_{i-1}) \to H^2(\O_S^{\+ h_i}) \cong \CC^{h_i}$ are isomorphisms and, by Proposition \ref{prop:stabile}, all $\E_i$ are stable with $H^2(\E_i)=0$ for $i=0, \ldots, N$. (We do not claim that there is a maximal such $N$; indeed, it may happen that the process does not terminate, i.e. all $h_i >0$, in which case any $N > 0$ fulfills the criteria.)

By \eqref{eq:univext} and properties \eqref{eq:p0} of $\E_0$, we have:
\begin{equation}
  \label{eq:p+} \rk \E_N =  r+1+h_1+\cdots+h_N, \; \; c_1(\E_N) =  [C], \; \; 
\chi(\E_N)  =  \chi(\E_0)=p-1-d=: \chi \geq 0.
\end{equation}
Since $\E_N$ is stable, it is simple, whence \eqref{eq:simple} and \eqref{eq:p+} yield: 
\begin{equation}
  \label{eq:ecco}
  p-1 \geq \chi \rk \E_N.
\end{equation}
The sequence \eqref{eq:univext} and coboundary maps $H^1(\E_{i-1}) \to H^2(\O_S^{\+ h_i})$ being isomorphisms yield
\begin{equation} \label{eq:hi}
h_{i+1}=h^1(\E_i) \geq 2h_i-h^0(\E_{i-1})=2h^1(\E_{i-1})-h^0(\E_{i-1})=h_i-\chi, \; \; \forall\, i \geq 1.
\end{equation}
In particular, we obtain that:
\begin{equation}
  \label{eq:acchi}
  h_i \geq h_1-(i-1)\chi \geq r+1-i\chi, \; \;  \forall\, i=1,\ldots, N.
\end{equation}
Hence, the procedure of taking extensions goes on at least until $i=N$, for any $N \geq 1$ satisfying
\begin{equation}
  \label{eq:N}
  N \chi \leq r.
\end{equation}
Note that if $N$ satisfies \eqref{eq:N}, then the inequalities \eqref{eq:ecco}
increase in strength as $N$ increases.

By \eqref{eq:p+} and \eqref{eq:acchi} we have:
\begin{eqnarray}
 \label{eq:rke} \rk \E_N & = & \sum_{i=1}^N h_i+r+1 \geq \sum_{i=0}^N (r+1-i\chi)  = (N+1)(r+1) -\chi \sum_{i=0}^N h_j \\
\nonumber
           & =& (N+1)(r+1) -\frac{ N(N+1)\chi}{2} = \Big(N+1\Big)\Big(r+1-\frac{N\chi}{2}\Big). 
\end{eqnarray}
Hence, we obtain from \eqref{eq:ecco} with $i=N$ that
\begin{equation}
  \label{eq:N>0}
p-1 \geq \chi \Big(N+1\Big)\Big(r+1-\frac{N\chi}{2}\Big),\; \mbox{or equivalently,} \; \;\rho(p,r,d) +r(r+2) \geq \chi N\Big(r+1-\frac{\chi(N+1)}{2}\Big).
\end{equation}

If $\chi >0$, then
the strongest inequality is obtained by using the largest $N$ satisfying \eqref{eq:N}, which is $N= \gamma:=\lfloor \frac{r}{\chi} \rfloor$. This proves \eqref{eq:rhocond} if $\chi>0$.

If $\chi=0$, then the left hand side of 
\eqref{eq:rhocondforte} is simply $d=p-1$, so we may assume that $r \geq 2$. The torsion free sheaf $\A':=\A \*\O_C(-P)$ for a general $P \in C$ is still globally generated and satisfies $\deg \A'=d-1=:d'$ and $r':=h^0(\A')-1=r$.
Since $1=p-1-d'<r',$ this falls into the lower line of \eqref{eq:rhocondforte}, which easily rewrites as the desired inequality 
$\rho(p,r,d) +r(r+2) \geq \frac{1}{2}r(r+1)$.
\end{proof}

\begin{remark} \label{rem:uffa}
We do not know if the stronger condition \eqref{eq:rhocondforte} is optimal, although it gets rid of the cases occurring in Examples \ref{cex1} and \ref{cex3}. One can easily verify that the inequality $d \geq r(r+1)$ is stronger than 
\eqref{eq:rhocondforte} in this range. We must however recall that our Theorem \ref{thm:main4} yields the existence of {\it birational} linear series on {\it smooth} curves; it is plausible that nonbirational linear systems, as well as torsion free sheaves on singular curves (cf. Example \ref{cex2}), may cover a wider value range of $p$, $r$ and $d$. 
We also remark that the torsion free sheaves in the quoted example satisfy equality in \eqref{eq:rhocondforte}.
\end{remark}

\begin{remark} \label{rem:h1>0}
The fact that bundles  $\E_{C,\A}$ with $h^1(\E_{C,\A}) >0$ do indeed exist when $0 \leq p-1-d \leq r$ follows from Theorem \ref{thm:main4}(ii). This shows that additional complexity arises for abelian surfaces in comparison with $K3$ surfaces, where the analogous bundles always have vanishing $H^1$. 
\end{remark}

\renewcommand{\proofname}{Proof}

\end{appendices}

%%%%%%%%%%%%%%%%%%%%%%%%%%%%%(BIBLIOGRAPHY)%%%%%%%%%%%%%%%%%%%%%%%%%%%%%%%
%
%
%%%%%%%%%%%%%%%%%%%%%%%%%%%%%%%%%%%%%%%%%%%%%%%%%%%%%%%%%%%%%%%%%%%%%%%

\end{document}